\documentclass[12pt,a4paper,reqno]{amsart}
\usepackage{amsmath}
\usepackage{amsfonts}
\usepackage{amssymb}
\usepackage{bm}
\usepackage{graphics}

\newcommand\R{{\mathbb{R}}}
\newcommand\C{{\mathbb{C}}}

\newcommand\D{{\mathbf{D}}}

\newcommand\I{{\mathbf{I}}}
\renewcommand\P{{\mathbf{P}}}
\newcommand\E{{\mathbf{E}}}

\renewcommand\Im{{\operatorname{Im}}}
\renewcommand\Re{{\operatorname{Re}}}
\newcommand\eps{{\varepsilon}}
\newcommand\trace{\operatorname{trace}}

\newcommand\dist{\operatorname{dist}}
\newcommand\Span{\operatorname{span}}
\newcommand\sgn{\operatorname{sgn}}

\renewcommand\a{x}
\renewcommand\b{y}

\newcommand\BBC{{\mathbb C}}

\def\x{{\bf X}}
\def\y{{\bf Y}}
\def\mb{\mbox}

\theoremstyle{plain}
  \newtheorem{theorem}[subsection]{Theorem}
  \newtheorem{conjecture}[subsection]{Conjecture}

  \newtheorem{proposition}[subsection]{Proposition}
  
  \newtheorem{lemma}[subsection]{Lemma}
  \newtheorem{corollary}[subsection]{Corollary}

\theoremstyle{remark}
  \newtheorem{remark}[subsection]{Remark}

\theoremstyle{definition}
  \newtheorem{definition}[subsection]{Definition}

\include{psfig}

\begin{document}

\title[Universality of ESDs and the circular law]{Random matrices:\\ Universality of ESDs and the circular law}

\author{Terence Tao}
\address{Department of Mathematics, UCLA, Los Angeles CA 90095-1555}
\email{tao@@math.ucla.edu}

\author{Van Vu}
\address{Department of Mathematics, Rutgers University, Piscataway NJ 08854-8019}
\email{vanvu@@math.rutgers.edu}

\author{Manjunath Krishnapur (Appendix)}
\address{Department of Mathematics, U. Toronto, Toronto Canada, MS5 2E4}
\email{manju@@math.toronto.edu}

\begin{abstract}
Given an $n \times n$ complex matrix $A$, let

$$\mu_{A}(x,y):=
\frac{1}{n} |\{1\le i \le n, \Re \lambda_i \le x, \Im \lambda_i \le
y\}|$$ be the empirical spectral distribution (ESD) of its
eigenvalues
 $\lambda_i \in \BBC, i=1, \dots n$.

 We consider the limiting distribution (both in probability and in the almost sure convergence sense)
 of the normalized ESD $\mu_{\frac{1}{\sqrt{n}} A_n}$
 of a random matrix $A_n = (a_{ij})_{1 \leq i,j \leq n}$
 where the random variables $a_{ij} - \E(a_{ij})$ are iid copies of a
  fixed random variable $x$ with unit variance.  We prove a
  \emph{universality principle} for such ensembles, namely that the
  limit distribution in question is {\it independent} of the actual choice of
  $x$. In particular, in order to compute this distribution, one can assume that
   $x$ is real of complex gaussian.  As a related result, we show how laws for this ESD follow from laws for the \emph{singular} value distribution of $\frac{1}{\sqrt{n}} A_n - zI$ for complex $z$.

As a corollary we establish  the Circular Law conjecture (both almost surely and in probability), that asserts that $\mu_{\frac{1}{\sqrt{n}} A_n}$ converges to the uniform measure on the unit disk when the $a_{ij}$ have zero mean.
\end{abstract}

\maketitle

\section{Introduction}


\subsection{Empirical spectral distributions}

This paper is concerned with the convergence of empirical spectral distributions of
random matrices, both in the sense of convergence in probability and in the almost sure sense.

\begin{definition}[Modes of convergence]\label{weak-def}  For each $n$,
let $F_n$ be a random variable taking values in some Hausdorff topological space $X$,
and let $F$ be another element of $X$.
\begin{itemize}
\item We say that $F_n$ \emph{converges in probability}
to $F$ if for every neighbourhood $V$ of $F$, we have $\lim_{n \to \infty} \P( F_n \in V )
 = 1$.
\item We say that $F_n$ \emph{converges almost surely}
to $F$ if we have $\P( \lim_{n \to \infty} F_n = F ) = 1$.
\end{itemize}
Similarly, if $X_n$ is a scalar random variable, we say that $X_n$
is \emph{bounded in probability} if we have
$$
\lim_{C \to \infty} \liminf_{n \to \infty} \P( |X_n| \leq C ) = 1
$$
and \emph{almost surely bounded} if we have
$$
\P( \limsup_n |X_n| < \infty ) = 1.
$$
\end{definition}

Let $M_n (\BBC)$ denote the set of $n \times n$ complex matrices.
For $A \in M_n(\BBC)$,  we let
$$\mu_{A}(s,t):=
\frac{1}{n} |\{1\le i \le n, \Re \lambda_i \le s, \Im \lambda_i \le
t \}|$$

be the \emph{empirical spectral distribution} (ESD) of its  eigenvalues
 $\lambda_i \in \BBC, i=1, \dots n$.
 This is a discrete probability measure on $\BBC$.

Now suppose that $A_n \in M_n(\BBC)$ is a random matrix ensemble
(i.e. a probability distribution on $M_n(\BBC)$), and let
$\mu_\infty$ be a probability measure on $\BBC$. We give the space of probability measures on $\BBC$ the usual \emph{vague topology}, thus a sequence of deterministic measures $\mu_n$ converges to $\mu$ if $\int_{\BBC} f\ d\mu_n$ converges to $\int_{\BBC} f\ d\mu$ for every \emph{test function} (i.e. continuous and compactly supported function) $f: \BBC \to \R$.  Thus, by Definition \ref{weak-def}, we see that $\mu_{\frac{1}{\sqrt{n}} A_n}$ converge in probability to $\mu_\infty$ if for
    every continuous and compactly supported function $f: \BBC \to \R$, the expression
\begin{equation}\label{fuz}
 \int_\BBC f(z)\ d\mu_{\frac{1}{\sqrt{n}} A_n}(z) - \int_\BBC f(z)\ d\mu_\infty
\end{equation}
converges to zero in probability, thus
$$ \lim_{n \to \infty} \P( |\int_\BBC f(z)\ d\mu_{\frac{1}{\sqrt{n}} A_n}(z) - \int_\BBC f(z)\ d\mu_\infty| \ge \eps ) = 0$$
for every $\eps > 0$.  Similarly, $\mu_{\frac{1}{\sqrt{n}} A_n}$ converges almost surely to $\mu_\infty$ if with probability $1$, the expression \eqref{fuz} converges to zero for all $f: \BBC \to \R$.

\begin{remark}
In practice, our matrices $A_n$ will have bounded entries on the
average, which suggests (by the Weyl comparision inequality, see
Lemma \ref{compar}) that their eigenvalues should be of size about $O(\sqrt{n})$; thus the
  normalization by $\frac{1}{\sqrt{n}}$ is natural.\end{remark}

\subsection{Universality}

A fundamental problem in the theory of random matrices is to
determine the limiting distribution of the ESD of a random matrix
ensemble (either in probability or in the almost sure sense), as the size of the random matrix tends to infinity.

The  situation  with this problem, so far,  is that the analysis
depends very much on which ensemble one is dealing with. In some
cases such as when the entries have  gaussian distribution, powerful
group-theoretic structure (e.g. invariance under the orthogonal
group $O(n)$ or unitary group $U(n)$) plays an essential role, as
one can use it to derive  an explicit formula for the joint
distribution of the eigenvalues. The limiting distribution can then
be computed directly from this formula. In the majority of cases,
however, there is little symmetry, and such a formula is not
available. Consequently, the problem becomes much harder and its
analysis typically requires tools from various areas of mathematics.

On the other hand, there is a well-known intuition behind this
problem (and many others concerning random matrices), the
\emph{universality} phenomenon, that asserts that the limiting
distribution should not depend on the particular distribution of the
entries. This phenomenon motivates many theorems and conjectures in
the area. In the following, we mention two famous examples, Wigner's
semi-circle law and the Circular Law conjecture.

{\it Wigner's semi circle law.} In the 1950's, motivated by
numerical experiments, Wigner \cite{wig} proved that the ESD of an
$n \times n$ hermitian matrix with (upper diagonal) entries being
iid gaussian random variables converge to the semi-circle law $F$ whose
density is given by

$$ \rho(x)= \begin{cases} \frac{1}{2\pi} \sqrt {4-x^2}, &|x| \le 2 \\ 0,
&|x| > 2. \end{cases} $$

Wigner's result (which holds for both modes of convergence) was later extended to many other ensembles. The most
general form only requires the mean and variance of the entries
\cite{Pas,bai}:

\begin{theorem} \label{theorem:SCL}
Let $A_n$ be the $n \times n$ hermitian random matrix whose upper
diagonal entries are iid complex random variables with mean 0 and
variance 1. Then the ESD of $\frac{1}{\sqrt n}A_n$ converges (both in probability and in the almost sure sense) to the
semi-circle distribution.
\end{theorem}

{\it Circular Law Conjecture.} The well-known Circular Law conjecture deals
with  non-hermitian matrices.

\begin{conjecture} \label{conj:CL}
Let $A_n$ be the $n \times n$  random matrix whose  entries are iid
complex random variables with mean 0 and variance 1. Then the ESD of
$\frac{1}{\sqrt n} A_n$ converges (both in probability and in the almost sure sense) to the uniform distribution on the
unit disk.
\end{conjecture}

Similarly to Wigner's law, this conjecture was posed, based on
numerical evidence, in the 1950's. The case when the
entries have complex gaussian distribution was verified by Mehta
\cite{meh} in 1967, using Ginibre's formula for the joint density
function of the eigenvalues of $A_n$ (see, for example,
\cite[Chapter 10]{bai}):

\begin{equation} \label{eqn:ginibre} p(\lambda_1, \dots, \lambda_n)=
c_n \prod_{i < j} |\lambda_i -\lambda_j|^2 \exp(-n \sum_{i=1}^n
|\lambda_i|^2). \end {equation}

Another case where such a formula is available is when the entries
have real gaussian distribution, and for this case the conjecture
was confirmed by Edelman \cite{edel}. For the general case when
there is no formula, the problem appears much harder. Important
partial results were obtained by Girko \cite{girko, twenty}, Bai
\cite{bai-first, bai}, and more recently G\"otze-Tikhomirov
\cite{gotze, gt2}, Pan-Zhou \cite{PZ} and the authors
\cite{TV-circular}. These results establish the conjecture (in almost sure or in probability forms) under
additional assumptions on the distribution $\a$. The strongest result in the previous literature is from \cite{TV-circular, gt2} in which the almost sure and in probability forms of the conjecture respectively were shown under the extra
assumption that the entries have finite $(2+\epsilon)$-th moment for
any positive constant $\epsilon$. An attempt to remove this extra
$\epsilon$ (and thus proving Conjecture \ref{conj:CL} in full
generality) was a motivation for this paper.

\vskip2mm
A demonstration of the circular law for the Bernoulli and the Gaussian case appears in Figure~\ref{figure:CircLaw}.
\begin{figure}
\centerline{\textbf{Bernoulli \hspace{1.8in} Gaussian }}
\begin{center}
\scalebox{.3}{\includegraphics{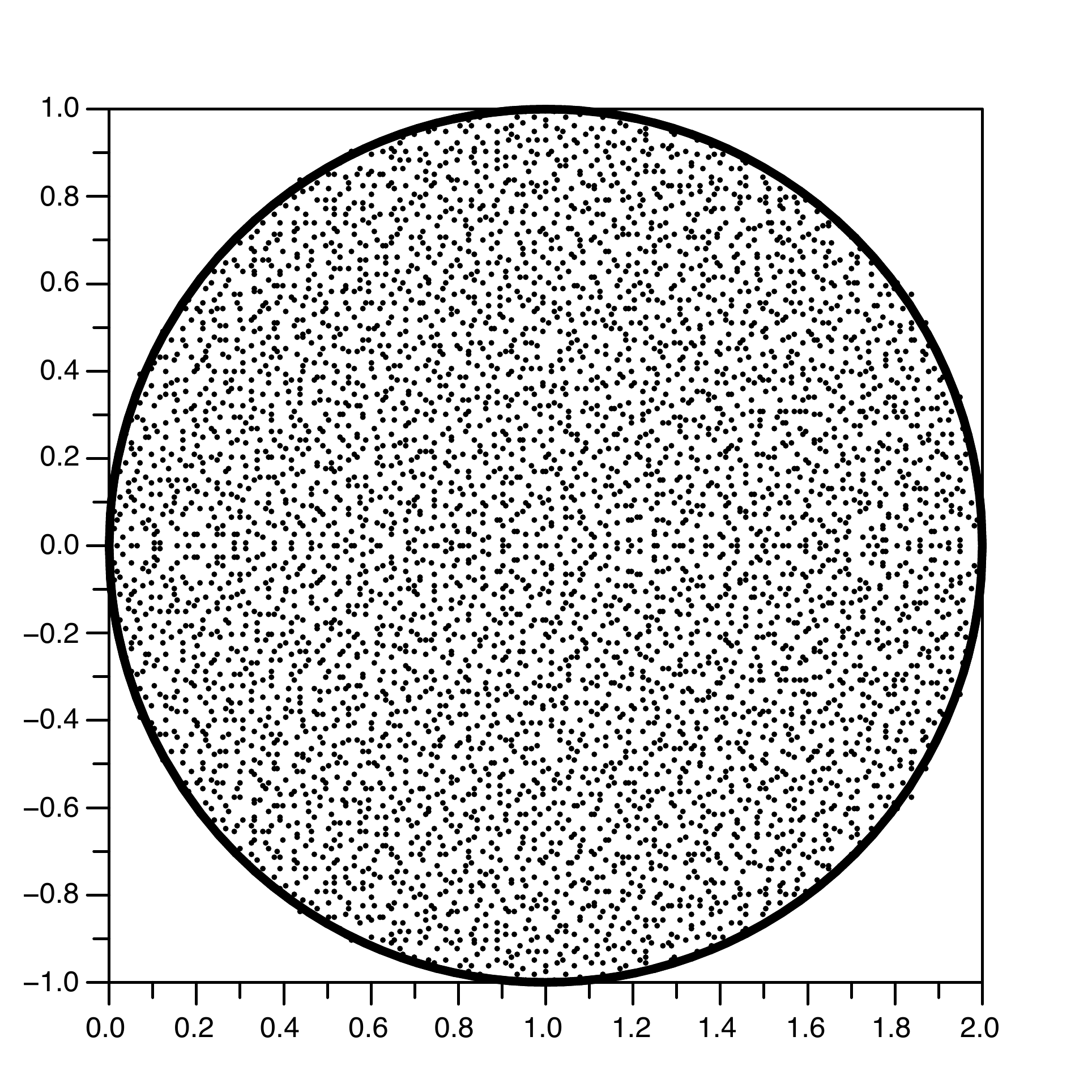}}
\scalebox{.3}{\includegraphics{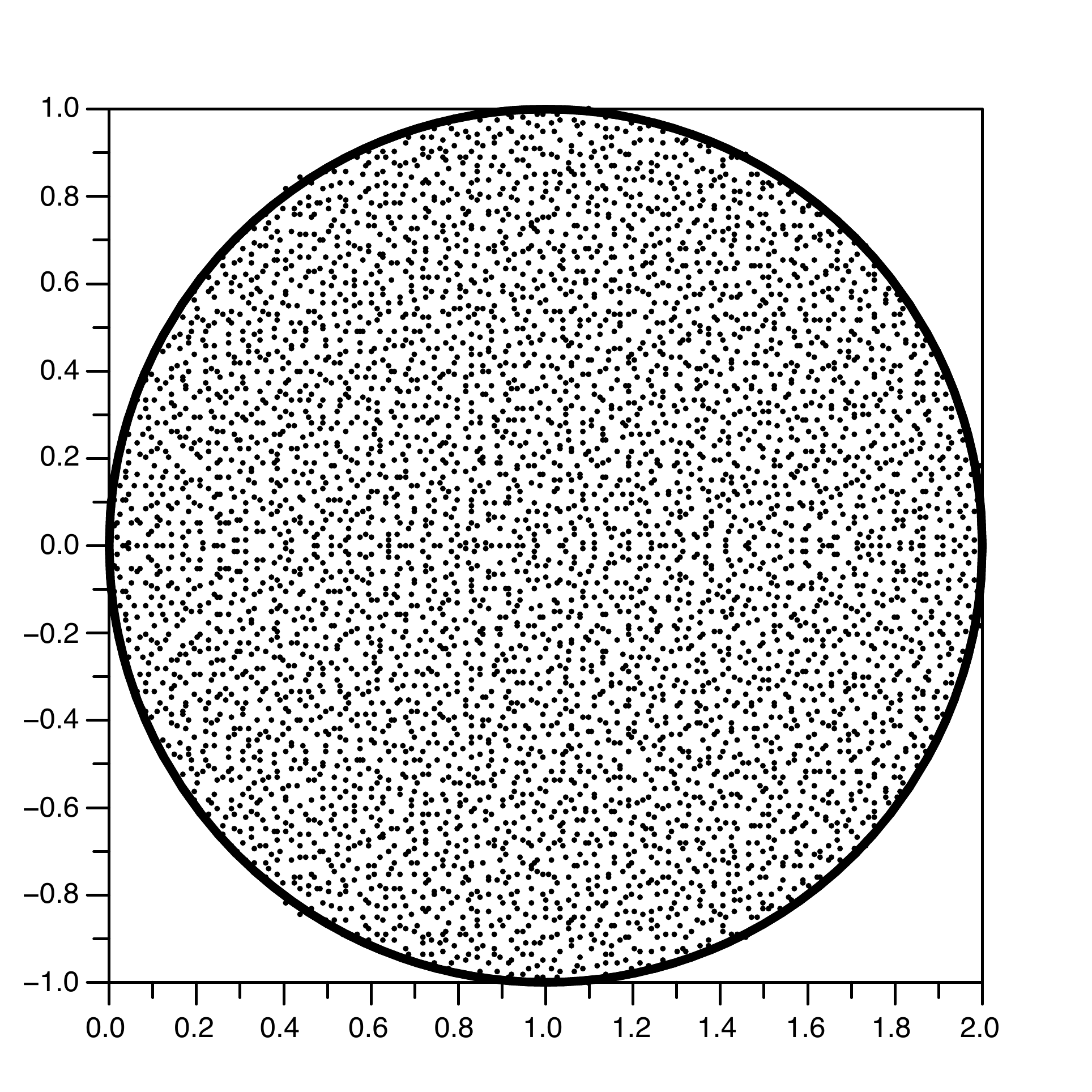}}
\end{center}
\caption{Eigenvalue plots of two randomly generated 5000 by 5000 matrices.  On the left, each entry was an iid Bernoulli random variable, taking the values $+1$ and $-1$ each with probability $1/2$.  On the right, each entry was an iid Gaussian normal random variable, with probability density function is $\frac{1}{\sqrt{2*\pi}} \exp( -x^2/2 ) $.  (These two distributions were shifted by adding the identity matrix, thus the circles are centered at $(1,0)$ rather than at the origin.)}
\label{figure:CircLaw}
\end{figure}

In both the semi-circular law and the circular law, we observe that only the mean and variance of the
entries play a role in the limiting distribution. This is a common
situation, in fact, for many other conjectures in random matrix
theory, such as Dyson's conjecture \cite[Chapter 1]{meh}, and this phenomenon
sometimes referred to as \emph{universality} in the literature.

In this paper, we rigorously  prove the universality phenomenon for
the ESD of random matrices. More precisely, we show that the
limiting distribution of the ESD of a random matrix ensemble $A_n$ depends
only the mean and variance of its entries, under a mild size condition on the mean $\E A_n$, and
under the assumption that the matrix $A_n - \E A_n$ has iid entries.

For any matrix $A$, we define the \emph{Hilbert-Schmidt norm} $\|A\|_2$ by the formula $\|A\| := \trace(AA^\ast)^{1/2} = \trace(A^\ast A)^{1/2}$.

\begin{theorem}[Universality principle]\label{theorem:main1}
Let $\a$ and $ \b$ be complex random variables with zero mean and
unit variance. Let $X_n =  (\a_{ij})_{1 \leq i,j \leq n}$ and $Y_n := (\b_{ij})_{1 \leq i,j \leq n}$ be $n \times n$ random matrices whose entries $\a_{ij}$, $\b_{ij}$ are iid copies of $\a$ and $\b$, respectively. For
each $n$, let $M_n$ be a deterministic $n \times n$ matrix
satisfying
\begin{equation} \label{eqn:conditionM}  \sup_n \frac{1}{n^2} \|M_n\|_2^2 < \infty. \end{equation}
Let $A_n:= M_n + X_n$ and $B_n:= M_n
+Y_n$. Then $\mu_{\frac{1}{\sqrt{n}} A_n} - \mu_{\frac{1}{\sqrt{n}}
B_n}$ converges in probability to zero.  If furthermore we make the additional hypothesis that the ESDs \begin{equation}\label{must}
\mu_{(\frac{1}{\sqrt{n}} M_n-zI) (\frac{1}{\sqrt{n}} M_n-zI)^\ast}
\end{equation}
converge to a limit for almost every $z$, then $\mu_{\frac{1}{\sqrt{n}} A_n} - \mu_{\frac{1}{\sqrt{n}} B_n}$ converges almost surely to zero.
\end{theorem}

\begin{remark} \label{remark:subsequence}
The theorem still holds if we restrict the size of the matrices to
an infinite subsequence $n_1 < n_2 < \dots $ of positive integers.  This freedom to pass to a subsequence is useful for technical reasons involving compactness arguments.
\end{remark}

The condition \eqref{eqn:conditionM} has the following useful consequence, which we shall use repeatedly:

\begin{lemma}[Tightness of ESDs]\label{tight} Let $M_n$ and $A_n$ be as in Theorem \ref{theorem:main1}.  Then the quantities $\frac{1}{n^2} \|A_n\|_2^2$ and $\int_{\BBC} |z|^2\ d\mu_{\frac{1}{\sqrt{n}} A_n}(z)$ are almost surely bounded (and hence also bounded in probability).
\end{lemma}

\begin{proof} By the Weyl comparison inequality (Lemma \ref{compar}) it suffices to show that
$\frac{1}{n^2} \|A_n\|_2^2$ is almost surely bounded.  By \eqref{eqn:conditionM} and the triangle inequality it suffices to show that $\frac{1}{n^2} \|X_n\|_2^2$ is almost surely bounded.  But this follows from the finite second moment of $\a$ and the strong law of large numbers.
\end{proof}

As an immediate corollary of Theorem \ref{theorem:main1}, we have

\begin{corollary}[Universality principle]\label{cor:main1}
Let $\a, \b$ be complex random variables with zero mean and unit
variance. Let $X_n$ and $Y_n$ be $n \times n$ random matrices whose
entries are iid copies of $\a$ and $\b$, respectively. For each $n$,
let $M_n$ be a deterministic $n \times n$ matrix satisfying
\eqref{eqn:conditionM}. Let $A_n:= M_n + X_n$ and $B_n:= M_n
+Y_n$. Then  if $\mu_{\frac{1}{\sqrt{n}} B_n} $ converges in probability  to a
limiting measure $\mu$, then $ \mu_{\frac{1}{\sqrt{n}} A_n}$ also
converges in probability to $\mu$.  If furthermore we make the additional hypothesis that the ESDs \eqref{must} converge to a limit for almost every $z$, then we can replace ``in probability'' by ``almost surely'' in the previous sentence.
\end{corollary}

A demonstration of this corollary appears in Figure~\ref{figure:MainThm}.
\begin{figure}
\centerline{\textbf{Bernoulli \hspace{1.8in} Gaussian }}
\begin{tabbing}
\qquad
\=
\scalebox{.3}{\includegraphics{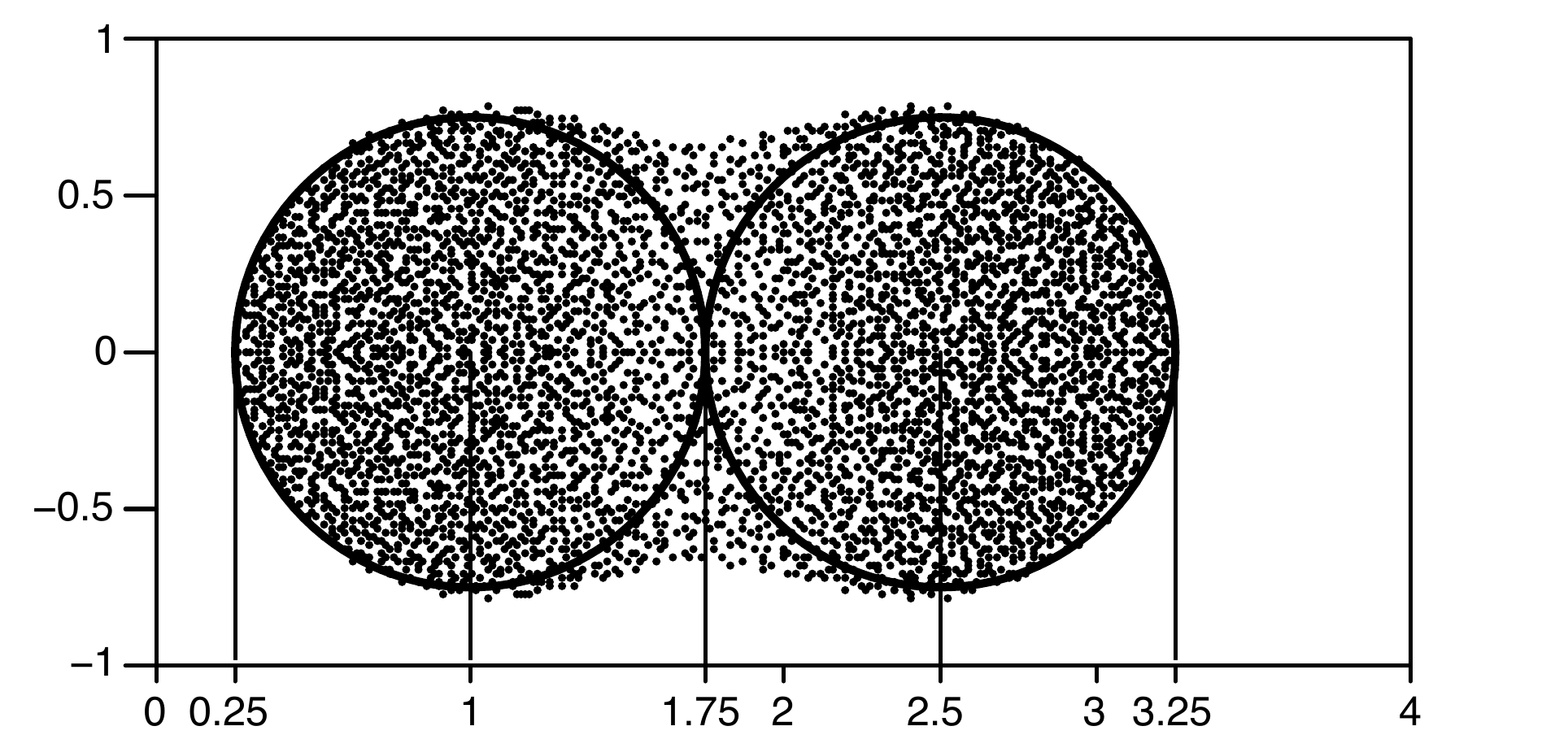}}
\=
\scalebox{.3}{\includegraphics{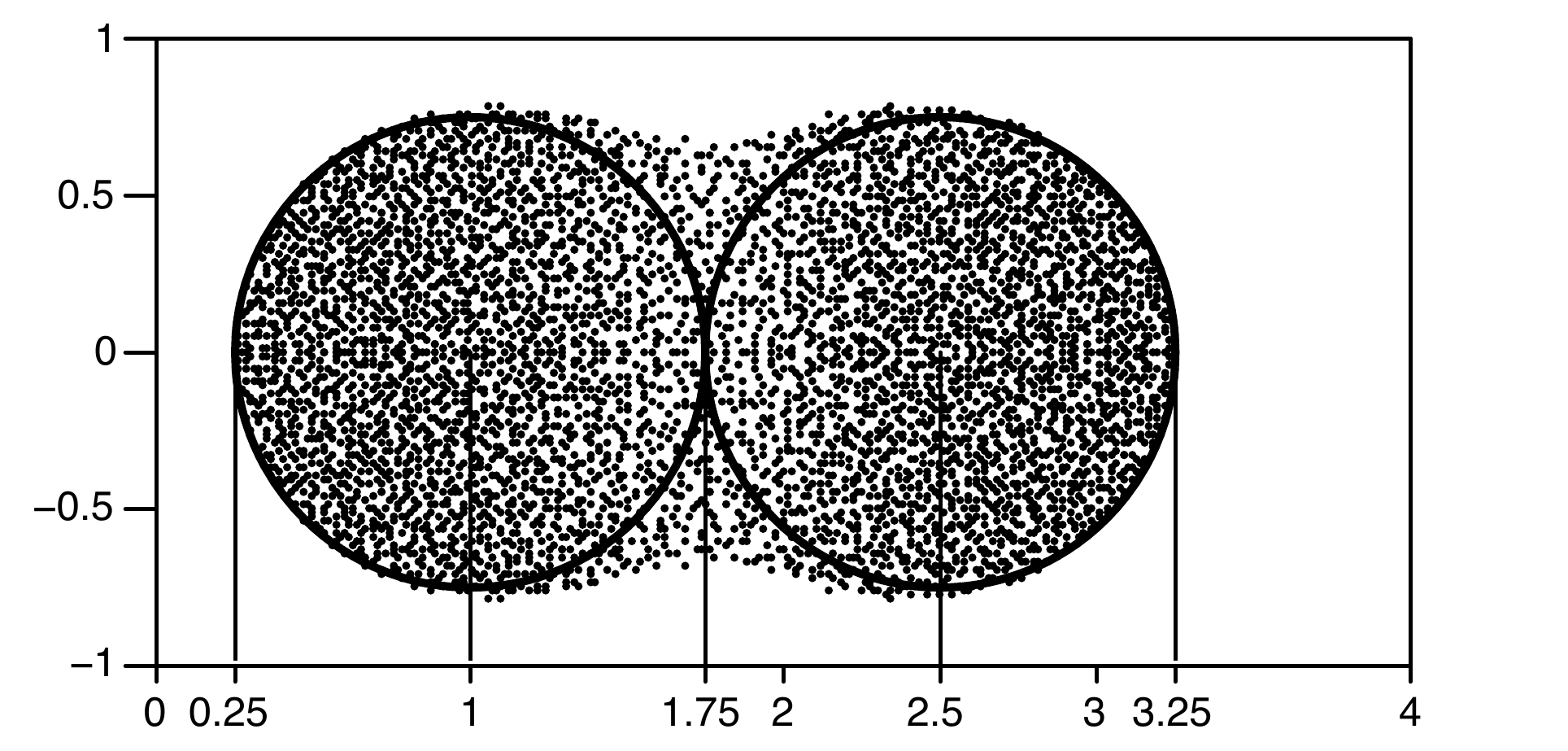}} \\
\>
\scalebox{.3}{\includegraphics{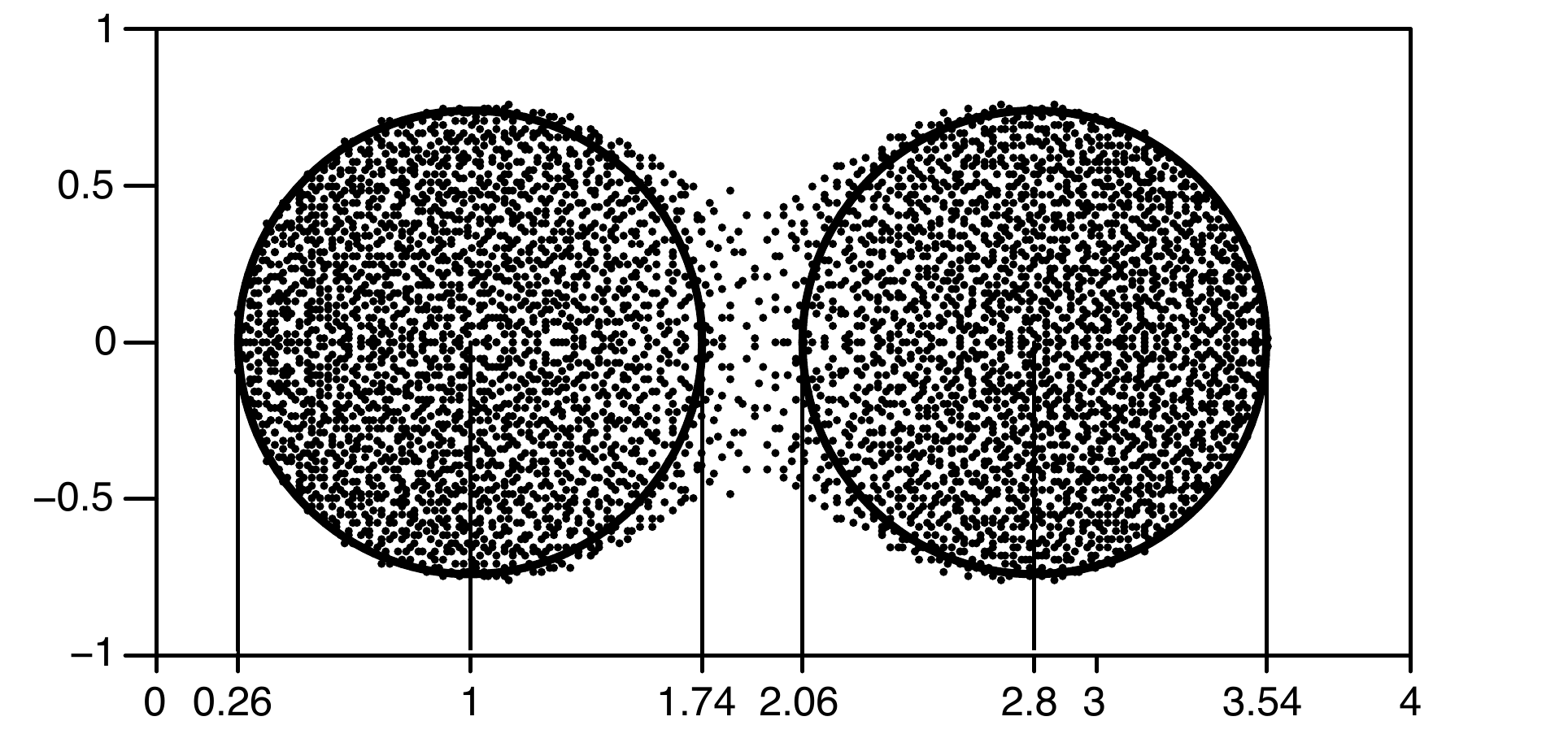}}
\>
\scalebox{.3}{\includegraphics{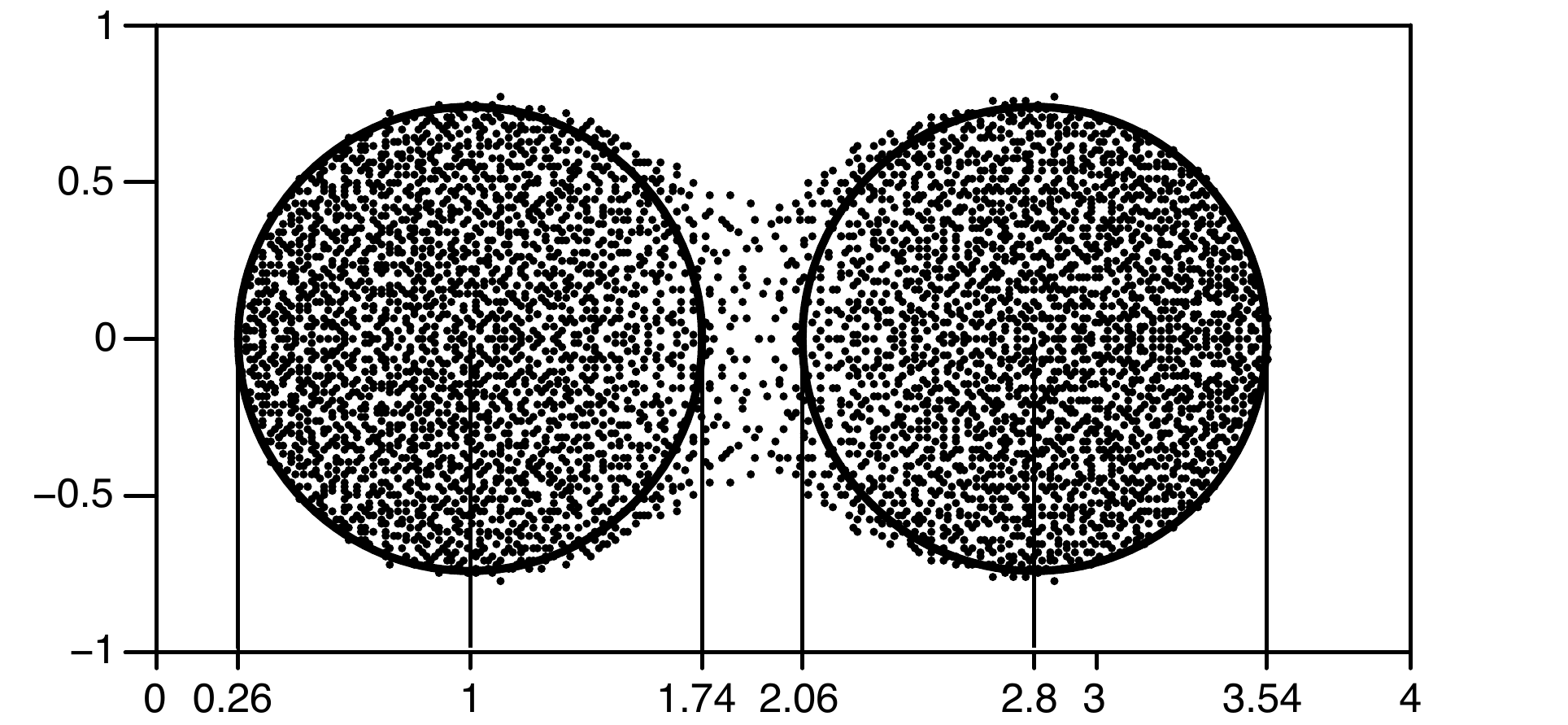}}
\end{tabbing}
\caption{Eigenvalue plots of randomly generated $n$ by $n$ matrices of the form $D_n+M_n$, where $n=5000$.  In left column, each entry of $M_n$ was an iid Bernoulli random variable, taking the values $+1$ and $-1$ each with probability $1/2$, and in the right column, each entry was an iid Gaussian normal random variable, with probability density function is $\frac{1}{\sqrt{2\pi}} \exp( -x^2/2 )$.  In the first row, $D_n$ is the deterministic matrix $\operatorname{diag}(1,1,\ldots,1,2.5,2.5,\ldots,2.5)$, and in the second row $D_n$ is the deterministic matrix $\operatorname{diag}(1,1,\ldots,1,2.8,2.8,\ldots,2.8)$ (in each case, the first $n/2$ diagonal entries are $1$'s, and the remaining entries are $2.5$ or $2.8$ as specified). }
\label{figure:MainThm}
\end{figure}

\begin{remark} One consequence of Corollary \ref{cor:main1} (in the case when \eqref{must} converges to a limit) is that the ESD $ \mu_{\frac{1}{\sqrt{n}} A_n}$ behaves asymptotically deterministically\footnote{The authors thank Oded Schramm for this observation.} in the sense that there exists a deterministic measure $\mu_n$ for each $n$ such that $\mu_{\frac{1}{\sqrt{n}} A_n} - \mu_n$ converges almost surely to zero. Indeed, one can simply take $\mu_n$ to be an instance of $\mu_{\frac{1}{\sqrt{n}} B_n}$, where the $B_n$ are selected independently of the $A_n$, and the claim will hold almost surely.  The question remains as to whether $\mu_n$ itself converges to some limit as $n \to \infty$; we partially address this issue in Theorem \ref{theorem:existence} below.
\end{remark}

\subsection{The Circular Law Conjecture}
Thanks to  Corollary \ref{cor:main1}, we can  reduce the problem of
computing the limiting distribution to the case when the entries are
gaussian\footnote{The idea of establishing a limiting law by first replacing a general random variable with a gaussian one is sometimes referred to as the ``Lindberg trick'' in the literature.} (or having any special distribution satisfying the variance
bound). In particular, since the Circular Law is verified
for random matrices with complex gaussian entries (see \cite{meh}), it follows that
this law (both in probability and in the almost sure sense) holds in full generality.  In other words, we have shown

\begin{theorem}[Circular Law]\label{theorem:CL} 
Let $X_n$ be the $n \times n$  random matrix whose  entries are iid
complex random variables with mean 0 and variance 1. Then the ESD of
$\frac{1}{\sqrt n} X_n$ converges (both in probability and in the almost sure sense) to the uniform distribution on the unit disk.
\end{theorem}

\begin{remark}
In \cite{TV-circular} (see also \cite{gt2} for an alternate proof for the in probability sense), this theorem was proven with the extra assumption that the entries have finite $(2+\eps)$-th moment
for any fixed $\eps >0$; earlier related results are appear in \cite{girko, twenty, bai-first, bai, gotze}.
\end{remark}

Notice that in Theorem \ref{theorem:CL}, we set $M_n$ to be the all
zero matrix (for which the boundedness and convergence hypotheses are trivial). In \cite{KV}, explicit distributions were computed for
the case when $M_n$ is an arbitrary diagonal matrix and $X_n$ has
iid gaussian entries. The formula for the limiting distribution is
somewhat technical, but its support is easy to describe: it is
exactly the set of $z \in \C$ for which $\int |z-x|^{-2} d\mu(x) \ge 1$ where
$\mu$ is the limiting distribution of the ESD of $M_n$. (In the
case $M_n$ is all zero, $\mu$ has all its mass at the origin, and so
the set of $z$ is the unit disk.)

The proof of Theorem \ref{theorem:main1} actually shows that if
$M_n$ and $M_n'$ both obey \eqref{eqn:conditionM} and have the property that the difference between the
ESD \eqref{must} and the counterpart for $M'_n$ converges to zero for almost every $z$, then
Theorem \ref{theorem:main1} holds with $A_n:= M_n + X_n$ and $B_n: =
M_n' + Y_n$ (see Remark \ref{remark:DS2}).

This has the following interesting consequence. Assume that  $M_n$ is a matrix with low rank, say $o(n)$. In this case, it is easy to see that the ESD \eqref{must} concentrates at $|z|^2$, since the matrix involved here is a self-adjoint low rank perturbation of $|z|^2 I$.  Thus, we can replace $M_n$ by the zero matrix and obtain

\begin{corollary} \label{cor:CL1} (Circular Law for shifted matrices)
Let $X_n$ be the $n \times n$  random matrix whose  entries are iid
complex random variables with mean 0 and variance 1 and $M_n$ be a deterministic
matrix with rank $o(n)$ and obeying \eqref{eqn:conditionM}. Let $A_n:= M_n + X_n$. Then the ESD of
$\frac{1}{\sqrt n} A_n$ converges (in either sense) to the uniform distribution on the unit disk.
\end{corollary}

In particular, it shows that Theorem \ref{theorem:CL} still holds if
the entries have (the same) non-zero mean. This extends a result of
Chafa\"i \cite{cha}, which in addition assumed that the entries
had finite fourth moment.


\subsection{Extensions}

We can extend Theorem \ref{theorem:main1}  in several ways. First,
by conditioning, we can obtain a theorem for $M_n$ being a random
matrix.

\begin{theorem}[Universality from a random base matrix]\label{theorem:main2}
Let $\a$ and $ \b$ be complex random variables with zero mean and
unit variance. Let $X_n = (\a_{ij})_{1 \leq i,j \leq n}$ and $Y_n = (\b_{ij})_{1 \leq i,j \leq n}$ be $n \times n$ random matrices
whose entries are iid copies of $\a$ and $\b$, respectively. For
each $n$, let $M_n$ be a random $n \times n$ matrix, independent of $X_n$ or $Y_n$, such that
$\frac{1}{n^2} \|M_n\|_2^2$ is bounded in probability (see Definition \ref{weak-def}).  Let $A_n:= M_n + X_n$ and $B_n:= M_n
+Y_n$. Then $\mu_{\frac{1}{\sqrt{n}} A_n} - \mu_{\frac{1}{\sqrt{n}}
B_n}$ converges in probability to zero.  If we furthermore assume that $\frac{1}{n^2} \|M_n\|_2^2$ is almost surely bounded, and \eqref{must} converges almost surely to some limit for almost every $z$, then $\mu_{\frac{1}{\sqrt{n}} A_n} - \mu_{\frac{1}{\sqrt{n}}
B_n}$ converges almost surely to zero.
\end{theorem}

We can also address a more general form of random matrices (cf. \cite{twenty}). Let
$K_n, L_n$ be two sequences of matrices. Define $A_n:= M_n+ K_n X_n
L_n$ and $B_n:= M_n + K_n Y_n L_n$. We can show that under some mild
assumptions on $M_n, K_n, L_n$, Theorem \ref{theorem:main1} still
holds:

\begin{theorem} \label{theorem:main3}
Let $\a$ and $ \b$ be complex random variables with zero mean and
unit variance. Let $X_n$ and $Y_n$ be $n \times n$ random matrices
whose entries are iid copies of $\a$ and $\b$, respectively. Let
$M_n, K_n, L_n$ be random $n \times n$ matrices (independent of $X_n, Y_n$) and let $A_n:= M_n +
K_n X_n L_n$ and $B_n:= M_n + K_n Y_n L_n$. Assume that the expressions
\begin{equation} \label{eqn:trace-3}
\frac{1}{n^2} \| A_n \|_2^2 + \frac{1}{n^2} \| B_n\|_2^2 + \frac{1}{n^2} \| K_n^{-1} M_n  L_n^{-1} \|_2^2 +
\frac{1}{n} \|K_n^{-1} L_n^{-1} \|_2^2
\end{equation}
are bounded in probability.  If furthermore we assume that \eqref{eqn:trace-3} is almost surely bounded, and that for almost every $z$ the ESDs
\begin{equation}\label{fancy-must}
\mu_{( \frac{1}{\sqrt{n}} K_n^{-1} M_n L_n^{-1} - z K_n^{-1} L_n^{-1}) ( \frac{1}{\sqrt{n}} K_n^{-1} M_n L_n^{-1} - z K_n^{-1} L_n^{-1})^\ast}
\end{equation}
converge almost surely to a limit, then $\mu_{\frac{1}{\sqrt{n}} A_n} - \mu_{\frac{1}{\sqrt{n}} B_n}$ converges almost surely to zero.
\end{theorem}

Note that Theorem \ref{theorem:main2} is the special case of Theorem \ref{theorem:main3} in which $K_n=L_n=I$. It seems of interest to see whether the hypotheses on \eqref{eqn:trace-3} can be verified for various natural random or deterministic matrices $M_n, K_n, L_n$, normalised appropriately by a suitable power of $n$.  We do not pursue this matter here.

A demonstration of the above theorem
for the Bernoulli and the Gaussian case appears in Figure~\ref{figure:Extension}.
\begin{figure}
\centerline{\textbf{Bernoulli \hspace{1.8in} Gaussian }}
\begin{center}
\scalebox{.32}{\includegraphics{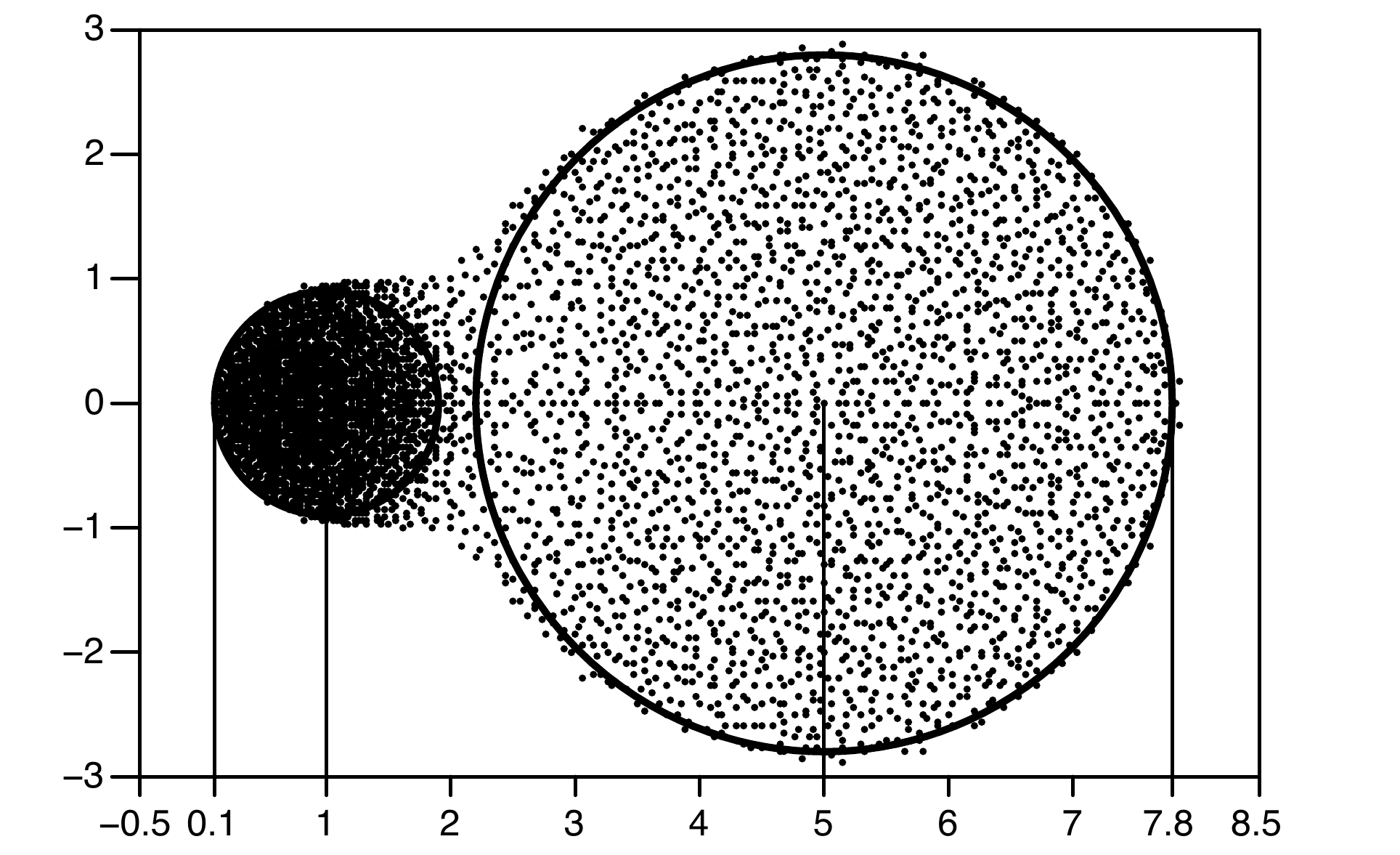}}
\scalebox{.32}{\includegraphics{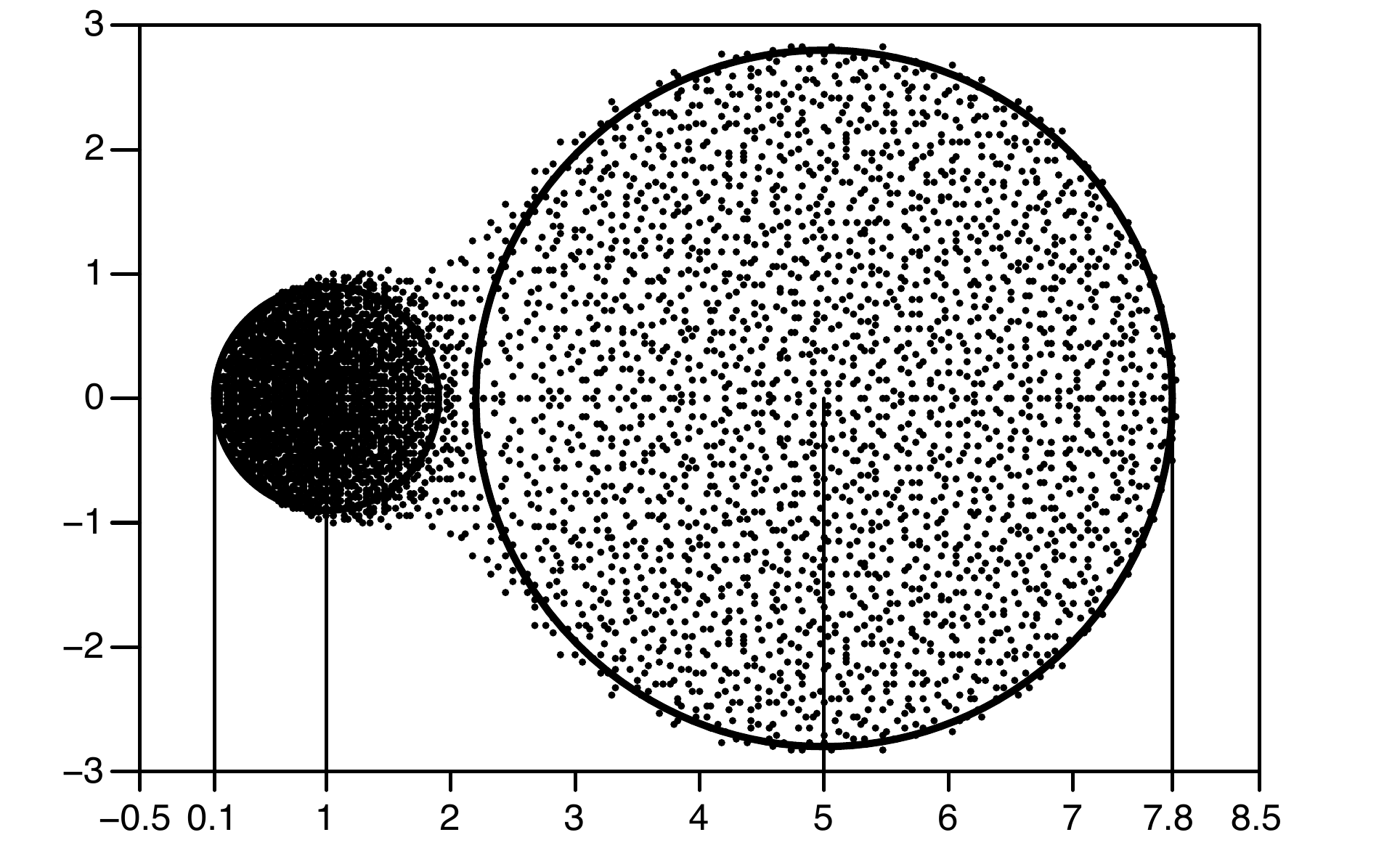}}
\end{center}
\caption{Eigenvalue plots of two randomly generated 5000 by 5000 matrices of the form $A + BM_nB$, where $A$ and$ B$ are diagonal matrices having $n/2$ entries with the value 1 followed by $n/2$ entries with the value 5 (for $D$) and the value $2$ (for $X$).
On the left, each entry of $M_n$ was an iid Bernoulli random variable, taking the values $+1$ and $-1$ each with probability $1/2$.  On the right, each entry of $M_n$ was an iid Gaussian normal random variable, with probability density function is $\frac{1}{\sqrt{2*\pi}} \exp( -x^2/2 ) $.}
\label{figure:Extension}
\end{figure}





The proofs of these extensions are discussed in Section
\ref{section:extensions}.

Another direction for generalization is to consider random matrices
whose entries are independent, but not necessarily identically
distributed.
 Most of the tools used in this paper (e.g. law of large numbers, Talagrand's inequality,
  and the least singular value bound from \cite{TV-circular}) extend without difficulty to this
  setting. Furthermore, Krishnapur pointed out  that one can also prove a ``universal'' version of
  Theorem \ref{theorem:DS}. This leads to a generalization in Appendix \ref{krish-sec} (written
  by Krishnapur).

For similar reasons, one expects to be able to extend the above results to the case
when $X_n$ and $Y_n$ are sparse iid random matrices; for instance,
the least singular value bounds from \cite{TV-circular} extend to this case,
 and the circular law for sparse iid matrices is already known in several cases \cite{gotze},
 \cite{TV-circular}. We, however,
will not pursue these matters here.

\subsection{Computing the ESD of a random non-hermitian
 matrix via the ESD of a hermitian one}

Theorem \ref{theorem:main1} provides one useful way to compute the
 (limiting distribution of) ESD of a random non-hermitian matrix, namely that one can restrict to
 any particular distribution (such as complex  gaussian) of the entries. The
 proof of this theorem (with some modification) also provides another way
 to deal with this problem, namely that one can reduce the problem
 of computing the ESD of $\frac{1}{\sqrt n}A_n$ to that of
 $(\frac{1}{\sqrt n} A_n -zI) (\frac{1}{\sqrt n} A_n -zI)^\ast$, for
 fixed $z\in \BBC$.  More precisely, we have the following equivalences.

\begin{theorem}[Equivalences for convergence]\label{theorem:main4} Let $A_n$ be as in Theorem
\ref{theorem:main1}, and let $\mu$ be a probability measure on $\BBC$ with the second moment condition $\int |z|^2\ d\mu(z) < \infty$.  Then the following are equivalent:
\begin{itemize}
\item[(i)] The ESD $\mu_{\frac{1}{\sqrt{n}} A_n}$ of $\frac{1}{\sqrt{n}} A_n$ converges in probability  to $\mu$.
\item[(ii)] For almost every complex number $z$, $\frac{1}{n} \log |\det( \frac{1}{\sqrt{n}} A_n - z I )|$ converges in probability to $\int_{\BBC} \log|w-z|\ d\mu(w)$.
\item[(iii)]  For almost every complex number $z$, there exists a sequence $\eps_n > 0$ of positive numbers converging to zero such that $\frac{1}{n} \log \det( ((\frac{1}{\sqrt{n}} A_n - z I) + \eps_n I) (\frac{1}{\sqrt{n}} A_n - z I)^\ast + \eps_n I )$ converges in probability to $2\int_{\C} \log|w-z|\ d\mu(w)$.
\end{itemize}
If furthermore the ESDs \eqref{must} converge to a limit for almost every $z$, then we can replace convergence in probability by almost sure convegence in the above equivalences.
\end{theorem}

We prove this result in Section \ref{thm4}.  As a corollary, we have a criterion for when $\frac{1}{\sqrt{n}} A_n$ converges to a distribution $\mu$:

\begin{corollary} \label{corollary:main4} Let $A_n$ be as in Theorem
\ref{theorem:main1}, and let $\mu$ be a probability measure on $\BBC$ with the second moment condition $\int |z|^2\ d\mu(z) < \infty$. Suppose that for almost every complex number $z$, the ESD of $(\frac{1}{\sqrt{n}} A_n - zI)(\frac{1}{\sqrt{n}} A_n - zI)^{\ast}$ converges in probability to a limiting distribution $\eta_z$ on $[0,+\infty)$ such that the integral $\int_{\BBC} \log t\ d\eta_z(t)$ is absolutely convergent and equal to $2\int_{\BBC} \log|w-z|\ d\mu(w)$.  Then the ESD of $\frac{1}{\sqrt{n}} A_n$ converges in probability to $\mu$.  If the ESDs \eqref{must} converge to a limit for almost every $z$, then we can replace convergence in probability by almost sure convergence in the above implication.
\end{corollary}

\begin{proof}  We verify the claim for almost sure convergence only; the proof for convergence in probability is similar and is left as an exercise to the reader.

By Lemma \ref{tight}, we see that for fixed $z$, $|\frac{1}{n} \trace(\frac{1}{\sqrt{n}} A_n - zI)(\frac{1}{\sqrt{n}} A_n - zI)^{\ast}|$ is also almost surely bounded.  Taking limits, we conclude that
$$ \int_{\BBC} t\ d\eta_z(t) < \infty.$$
We then see from the dominated convergence theorem that for any $\eps > 0$, $\frac{1}{n} \log \det( ((\frac{1}{\sqrt{n}} A_n - z I)+\eps I) (\frac{1}{\sqrt{n}} A_n - z I)^\ast + \eps I )$ converges almost surely to $\int_{\BBC} \log( t+\eps )\ d\eta_z(t)$.  From this we obtain hypothesis (iii) of Theorem \ref{theorem:main4} (if $\eps_n$ is chosen to decay to zero sufficiently slowly), and the claim follows.
\end{proof}

Since the eigenvalues of $(\frac{1}{\sqrt{n}} A_n -
zI)(\frac{1}{\sqrt{n}} A_n - zI)^{\ast}$ are the squares of the
singular values of $\frac{1}{\sqrt n} A_n -zI$, we can also say that
Theorem \ref{theorem:main4} reduces the problem of computing the
limiting distribution of the eigenvalues of $\frac{1}{\sqrt n} A_n$
to that of the singular values of $\frac {1}{\sqrt n} A_n- zI$.

The big gain here is that the matrix $(\frac{1}{\sqrt{n}} A_n -
zI)(\frac{1}{\sqrt{n}} A_n - zI)^{\ast}$ is hermitian. (Random
matrices of this type are often called \emph{sample covariance matrices} in
the literature.) This allows one to use standard tools such as
truncation, Wigner's moment method and Stieljes transform (see, for
instance, the proof of Theorem \ref{theorem:SCL}
in \cite[Chapter 2]{bai}), or results such as Theorem \ref{theorem:DS}; techniques from free probability are also very powerful for such problems. These methods cannot be applied  to non-hermitian matrices for various reasons (see \cite[Chapter
10]{bai} for a discussion) and their failure has been the main difficulty in attacking problems
such as the Circular Law conjecture.

One can use Corollary \ref{corollary:main4} to give another proof of
Theorem \ref{theorem:CL}, without relying on  explicit formulas such
as \eqref{eqn:ginibre}. We omit the details.

\subsection {Existence of the limit} The results in the previous
chapters provide two different ways to compute (explicitly) the
limiting measure of the ESD of random matrices. In fact there is a simple compactness argument that  guarantees the existence of the limit, assuming of course that the deterministic ESDs \eqref{must} already converge, although the argument does not provide too much information on what the limit actually is.  More precisely, we have

\begin{theorem} \label{theorem:existence}
Let $\a$  be a complex random variable with zero mean and unit
variance. Let $X_n$ be the $n \times n$ random matrix whose entries
are iid copies of $\a$. For each $n$, let $M_n$ be a deterministic
$n \times n$ matrix satisfying
\begin{equation}\label{eqn:conditionM-2-deterministic}
\sup_n \frac{1}{n^2} \|M_n\|_2^2 < \infty.
\end{equation}
Assume furthermore that the ESD \eqref{must} converges for almost every $z \in \C$. Then the ESD of
$\frac{1}{\sqrt n} A_n$, where $A_n:= M_n + X_n$, converges (in both senses) to a limiting measure $\mu$.
\end{theorem}

\begin{proof}
We let $f_1, f_2, f_3, \ldots$ be an enumeration of a sequence of test functions which is dense in the uniform topology (such a sequence exists thanks to the Stone-Weierstrass theorem and the compact support of test functions).  By applying the Bolzano-Weierstrass theorem once for each function in this sequence and then using the Arzel\'a-Ascoli diagonalization argument, we can refine the subsequence so that $ \int_\BBC f_j(z)\ d\mu_{\frac{1}{\sqrt{n}} A_n}(z)$ converges in probability to some limit for each $j$, and hence by a limiting argument $ \int_\BBC g(z)\ d\mu_{\frac{1}{\sqrt{n}} A_n}(z)$ converges in probability to a limit for each test function $g$.  By the Riesz representation function we conclude that along this subsequence, $\mu_{\frac{1}{\sqrt{n}} A_n}$ converges in probability to some limit $\mu$, which is also a probability measure by the tightness bounds in Lemma \ref{tight}.

Applying Theorem \ref{theorem:main4}, we conclude that for almost every $z$, the expression
\begin{equation}\label{weaklim}
\frac{1}{n} \log \det( ((\frac{1}{\sqrt{n}} A_n - z I)+\eps_n I) ((\frac{1}{\sqrt{n}} A_n - z I)^\ast + \eps_n I ))
\end{equation}
converges in probability to $2\int_{\BBC} \log|w-z|\ d\mu(w)$ along this sequence, for some $\eps_n$ converging to zero.  On the other hand, from the hypotheses and the theorem of Dozier and Silverstein (see Theorem \ref{theorem:DS}) we know that for almost every $z$, the expression \eqref{weaklim} has a almost sure limit for the entire sequence of $n$.  Combining the two facts we see  that for almost every $z$, \eqref{weaklim} in fact converges almost surely to $2\int_{\BBC} \log|w-z|\ d\mu(w)$ for all $n$.  The claim now follows from another application of Theorem \ref{theorem:main4}.
\end{proof}

\subsection{Notation} The asymptotic notation is used under the
assumption that $n \rightarrow \infty$, holding all other parameters fixed.  Thus for instance, if we say that a quantity $a_{z,n}$ depending on $n$ and another parameter $z$ is equal to $o(1)$, this means that $a_{z,n}$ converges to zero as $n \to \infty$ for fixed $z$, but this convergence need not be uniform in $z$.  As another example, the condition \eqref{eqn:conditionM} is equivalent to asserting that $\|M_n\| = O(n)$ as $n \to \infty$.

\section{The replacement principle}

The first step toward Theorem \ref{theorem:main1} is the following
result that gives a general criterion for two random matrix
ensembles $\frac{1}{\sqrt{n}} A_n, \frac{1}{\sqrt{n}} B_n$ to
converge to the same limit.

\begin{theorem}[Replacement principle]\label{theorem:replacement} Suppose for each $n$
that $A_n, B_n \in M_n(\BBC)$ are ensembles of random matrices.
  Assume that
\begin{itemize}
\item[(i)] The expression
\begin{equation}\label{pan}
\frac{1}{n^2} \|A_n\|_2^2 + \frac{1}{n^2} \|B_n\|_2^2
\end{equation}
is bounded in probability (resp. almost surely).
\item[(ii)] For almost all complex numbers
$z$, $$\frac{1}{n}
 \log |\det(\frac{1}{\sqrt{n}} A_n - zI)| - \frac{1}{n} \log |\det(\frac{1}{\sqrt{n}} B_n -
 zI)|$$
converges in probability (resp. almost surely) to zero.  In particular, for each fixed $z$, these determinants are non-zero with probability $1-o(1)$ for all $n$ (resp. almost surely non-zero for all but finitely many $n$).
\end{itemize}
Then $\mu_{\frac{1}{\sqrt{n}} A_n} - \mu_{\frac{1}{\sqrt{n}} B_n}$ converges in probability (resp. almost surely) to zero.
\end{theorem}

We would like to remark here that we do not need to require
independence among the entries of $A_n$ and $B_n$.
  The proof of this theorem is rather ``soft'' in nature, relying primarily on
the Stieltjes transform technique (following Girko \cite{girko})
that analyses the ESD $\mu_{\frac{1}{\sqrt{n}} A_n}$ in terms of the
log-determinants $\frac{1}{n} \log |\det(\frac{1}{\sqrt{n}} A_n -
zI)|$, combined with tools from classical real analysis such as the
dominated convergence theorem (see Lemma \ref{dom} for the precise
version of this theorem that we need). The details are given in
Section \ref{section:replacement}.

In view of Lemma \ref{tight}, we see that Theorem \ref{theorem:main1}
follows immediately from Theorem \ref{theorem:replacement}
 and the following proposition.

\begin{proposition}[Converging determinant]\label{lemma:determinant}
Let $\a$ and $ \b$ be complex random variables with zero mean and
unit variance. Let $X_n$ and $Y_n$ be $n \times n$ random matrices
whose entries are iid copies of $\a$ and $\b$, respectively. For
each $n$, let $M_n$ be a deterministic $n \times n$ matrix
satisfying \eqref{eqn:conditionM}.
Set $A_n:= M_n +X_n$ and $B_n:= M_n
+ Y_n$.
Then for every fixed $z \in \BBC$,
\begin{equation}\label{aziz}
\frac{1}{n} \log |\det(\frac{1}{\sqrt{n}} A_n - zI)| - \frac{1}{n}
\log |\det(\frac{1}{\sqrt{n}} B_n - zI)|
\end{equation}
converges in probability to zero.  If furthermore we assume that \eqref{must} converges to a limit for this value of $z$, then \eqref{aziz} converges almost surely to zero.
\end{proposition}

For any square matrix $A$ of size $n$, let $\lambda_i(A)$ and $s_i(A)$
be the eigenvalues and singular values of $A$. Furthermore, let
$d_i(A)$ be the distance from the $i$th row vector of $A$ to the
subspace formed by the first $i-1$ row vectors. From linear algebra,
we have the fundamental identity

\begin{equation}\label{det}
|\det A| =\prod_{i=1}^n |\lambda_i(A)| = \prod_{i=1}^n s_i (A) =
\prod_{i=1}^n d_i (A).
\end{equation}

We will need to study the singular values and distances of
$\frac{1}{\sqrt{n}} A_n - zI$ and $\frac{1}{\sqrt{n}} B_n - zI$ in
order to estimate their determinants.  The proof of Proposition
\ref{lemma:determinant}, which occupies  Sections \ref{detsec},
\ref{section:highdim} and \ref{section:lowdim}, is the heart of the
paper. This proof relies on the following three ingredients:

\begin{itemize}
\item A result by  Dozier and Silverstein \cite{doz} that compares the ESD of the singular values of the matrices
$\frac{1}{\sqrt{n}} A_n - zI$ and $\frac{1}{\sqrt{n}} B_n - zI$. This will let us handle all the rows from $1$ to $(1-\delta) n$ for some small $\delta > 0$.
 \item
 A lower tail estimate for the distance between a random vector
  and a fixed subspace of relatively large co-dimension, using a concentration inequality of Talagrand
  \cite{Le}.
This will  handle the contribution of the rows between $(1-\delta)n$
and (say)  $n-n^{0.99}$.

  \item A polynomial lower bound for the least
    singular value of $\frac{1}{\sqrt{n}} A_n - zI$ and
     $\frac{1}{\sqrt{n}} B_n - zI$ from \cite{TV-circular, TV-lsv}.
     This bound enables us  to handle the contribution of the last $n^{0.99}$ rows.
\end{itemize}

\section{The replacement principle} \label{section:replacement}

The purpose of this section is to establish Theorem \ref{theorem:replacement}.
We begin with a version of the dominated convergence theorem.

\begin{lemma}[Dominated convergence]\label{dom}  Let $(X,\nu)$
be a finite measure space.  For each integer $n \geq 1$,
let $f_n: X \to \R$ be a random functions which are
 jointly measurable with respect to $X$ and the underlying probability space.  Assume that
\begin{itemize}
\item[(i)] (Uniform integrability) There exists $\delta > 0$ such
that $\int_X |f_n(x)|^{1+\delta}\ d\nu$ is bounded in probability (resp. almost surely).
\item[(ii)] (Pointwise convergence in probability) For $\nu$-almost every $x \in X$,
$f_n(x)$ converges in probability (resp. almost surely) to zero.
\end{itemize}
Then $\int_X f_n(x)\ d\nu(x)$ converges in probability (resp. almost surely) to zero.
\end{lemma}

\begin{proof}  We first prove the claim for convergence in probability.
We can normalise $\nu$ to be a probability measure.  Let $\eps > 0$ be arbitrary.  It suffices to show that
$$ \int_X f_n(x)\ d\nu(x) = O(\eps)$$
with probability $1-O(\eps)-o(1)$.

By hypothesis (i), we already know that with probability $1-O(\eps)-o(1)$, that
$$ \int_X |f_n(x)|^{1+\delta}\ d\nu(x) \leq C_\eps$$
for some $C_\eps$ depending on $\eps$.
This implies that
$$ \int_X f_n(x) \I( |f_n(x)| \geq M )\ d\nu(x) \leq C_\eps / M^\delta$$
for any $M > 0$, where $\I(E)$ denotes the indicator of an event $E$.  In particular, for $M$ large enough we have
$$ \int_X f_n(x) \I( |f_n(x)| \geq M )\ d\nu(x) \leq \eps,$$
with probability $1-O(\eps)-o(1)$, and so it will suffice to show that
\begin{equation}\label{so}
\int_X f_n(x) \I( |f_n(x)| \leq M )\ d\nu(x) = O(\eps)
\end{equation}
with probability $1-o(1)$.

Fix $M$.  By hypothesis, we have $\lim_{n \to \infty} \P( |f_n(x)| \geq \eps ) = 0$ for $\nu$-almost every $x \in X$.  By the dominated convergence theorem, we conclude that
$$ \int_X \P( |f_n(x)| \geq \eps )\ d\nu(x) = o(1).$$
By Fubini's theorem, we conclude that
$$ \E \int_X \I( |f_n(x)| \geq \eps )\ d\nu(x) = o(1)$$
and so by Markov's inequality, we have
$$ \int_X \I( |f_n(x)| \geq \eps )\ d\nu(x) = O(\eps/M)$$
with probability $1-o(1)$.  The claim \eqref{so} easily follows.

Now we prove the claim for almost sure convergence.  Again we let $\nu$ be a probability measure and $\eps > 0$ be arbitrary.  With probability $1-O(\eps)$ we have
$$ \int_X |f_n(x)|^{1+\delta}\ d\nu(x) \leq C_\eps$$
for all sufficiently large $n$, and some $C_\eps$ depending on $n$.  Also, with probability $1$, $f_n(x)$ converges to zero for almost every $x$.  The claim now follows by invoking (the deterministic special case of) the convergence in probability version of the lemma that we have just proven.
\end{proof}

Now we begin the proof of Theorem \ref{theorem:replacement}. We thus assume that $A_n, B_n$ are as in that theorem.  We shall first prove the claim for convergence in probability, and indicate later how to modify the proof to obtain the principle for almost sure convergence.

From the boundedness in probability of \eqref{pan} and Weyl's comparison inequality (Lemma \ref{compar}) we see that for every $\eps > 0$ there exists $C_\eps > 0$ such that for each $n$, the eigenvalues $\lambda_1,\ldots,\lambda_n$ of $A_n$ obey the bound
\begin{equation}\label{lam}
 \sum_{j=1}^n \frac{1}{n^2} |\lambda_j|^2 \leq C_\eps
 \end{equation}
or equivalently that
$$ \int_\BBC |z|^2\ d\mu_{\frac{1}{\sqrt{n}} A_n}(z) \leq C_\eps$$
with probability $1-O(\eps)-o(1)$. Similarly we have
$$ \int_\BBC |z|^2\ d\mu_{\frac{1}{\sqrt{n}} B_n}(z) \leq C_\eps.$$
In particular, for each $n$ we see that with probability $1-O(\eps)-o(1)$ we have the tightness bounds
\begin{equation}\label{tight-1}
 \mu_{\frac{1}{\sqrt{n}} A_n} \{ z \in \BBC: |z| \geq R \} \leq C_\eps/R^2
 \end{equation}
and
\begin{equation}\label{tight-2}
\mu_{\frac{1}{\sqrt{n}} B_n} \{ z \in \BBC: |z| \geq R \} \leq C_\eps/R^2
\end{equation}
for all $R > 0$.

We now take the standard step of passing from the ESDs
$\mu_{\frac{1}{\sqrt{n}} A_n}, \mu_{\frac{1}{\sqrt{n}} B_n}$ to the
characteristic functions $m_{\frac{1}{\sqrt{n}} A_n},
m_{\frac{1}{\sqrt{n}} B_n}: \R^2 \to \BBC$, which are defined by the
formulae
\begin{align*}
m_{\frac{1}{\sqrt{n}} A_n}(u,v) &:= \int_\BBC e^{i u \Re(z) + i v \Im(z)}\ d\mu_{\frac{1}{\sqrt{n}} A_n}(z) \\
m_{\frac{1}{\sqrt{n}} B_n}(u,v) &:= \int_\BBC e^{i u \Re(z) + i v
\Im(z)}\ d\mu_{\frac{1}{\sqrt{n}} B_n}(z)
\end{align*}
thus the functions $m_{\frac{1}{\sqrt{n}} A_n}, m_{\frac{1}{\sqrt{n}} B_n}$ are continuous
 and are bounded uniformly in magnitude by $1$.

Thanks to the tightness bounds \eqref{tight-1}-\eqref{tight-2}, we can easily pass back and forth between convergence of ESDs and convergence of characteristic functions:

\begin{lemma}\label{probconv} Let the notation and assumptions be as above.  Then the following are equivalent:
\begin{itemize}
\item[(i)] $\mu_{\frac{1}{\sqrt{n}} A_n} - \mu_{\frac{1}{\sqrt{n}} B_n}$ converges in probability.
\item[(ii)] For almost every $u,v$, $m_{\frac{1}{\sqrt{n}} A_n}(u,v) - m_{\frac{1}{\sqrt{n}} B_n}(u,v)$ converges in probability.
\end{itemize}
\end{lemma}

\begin{proof}  We first show that (i) implies (ii).  Fix $u,v$, and let $\eps > 0$ be arbitrary.
From \eqref{tight-1}, \eqref{tight-2} we can find an $R$ depending on $C_\eps$ and $\eps$ such that
$$ \mu_{\frac{1}{\sqrt{n}} A_n}(\{ z \in \BBC: |z| \geq R \})+
\mu_{\frac{1}{\sqrt{n}} B_n}(\{ z \in \BBC: |z| \geq R \}) \leq \eps$$
with probability $1-O(\eps)-o(1)$.  In particular, with probability $1-O(\eps)-o(1)$ we have
$$
m_{\frac{1}{\sqrt{n}} B_n}(u,v) - m_{\frac{1}{\sqrt{n}} A_n}(u,v) = \int \psi(z/R) e^{i u \Re(z) + i v \Im(z)}\ [d\mu_{\frac{1}{\sqrt{n}} B_n}(z) - d\mu_{\frac{1}{\sqrt{n}} A_n}(u,v)(z)] + O(\eps)$$
where $\psi$ is any smooth compactly supported function that equals one on the unit ball.
But since $\mu_{\frac{1}{\sqrt{n}} B_n} - \mu_{\frac{1}{\sqrt{n}} A_n}$ converges in probability, the integral here converges to zero in probability.  The claim follows.

Now we prove that (ii) implies (i).  Since continuous compactly
supported functions are the uniform limit of smooth compactly
supported functions, it suffices to show that $\int_\BBC f\
d\mu_{\frac{1}{\sqrt{n}} A_n} - \int_\BBC f\ d\mu_{\frac{1}{\sqrt{n}} B_n}$ converges in probability to zero for
every smooth compactly supported function $f: \BBC \to \BBC$.

Now fix a smooth compactly supported function $f: \BBC \to \BBC$.
By Fourier analysis, we can write
\begin{equation}\label{mufan}
\int_\BBC f\ d\mu_{\frac{1}{\sqrt{n}} A_n} - \int_\BBC f\ d\mu_{\frac{1}{\sqrt{n}} B_n} =
\int_\R \int_\R \hat f(u,v) (m_{\frac{1}{\sqrt{n}} A_n}(u,v) -
m_{\frac{1}{\sqrt{n}} B_n}(u,v))\ du dv
\end{equation}
for some smooth, rapidly decreasing function $\hat f$.  In particular, the measure $d\nu = \hat f(u,v)\ du dv$ is finite.  The claim now follows from dominated convergence (Lemma \ref{dom}); note that the function $m_{\frac{1}{\sqrt{n}} A_n} - m_{\frac{1}{\sqrt{n}} B_n}$ is bounded and so clearly obeys the moment condition required in that lemma.
\end{proof}

In view of the above lemma, it suffices to show that
 $m_{\frac{1}{\sqrt{n}} A_n}(u,v) - m_{\frac{1}{\sqrt{n}} B_n}(u,v)$
 converges in probability to zero for almost every $u,v \in \R$.

Fix $u,v$. Since we can exclude a set of measure zero, we can assume
that $u,v$ are non-zero.  We allow all implied constants in the
arguments below to depend on $u,v$.

Following Girko \cite{girko}, we now proceed via the Stieltjes-like
transform $g_{\frac{1}{\sqrt{n}} A_n}: \BBC \to \R$, defined almost
everywhere by the formula
\begin{equation}\label{gobs}
\begin{split}
 g_{\frac{1}{\sqrt{n}} A_n}(z) &:= 2\Re \int_\BBC \frac{z-w}{|z-w|^2}\ d\mu_{\frac{1}{\sqrt{n}}}(w)\\
 &= \frac{2}{n} \Re \sum_{j=1}^n \frac{z - \frac{1}{\sqrt{n}} \lambda_j}{|z - \frac{1}{\sqrt{n}} \lambda_j|^2};
\end{split}
\end{equation}
observe that this is a locally integrable function on $\BBC$, and
that
\begin{equation}\label{snort}
g_{\frac{1}{\sqrt{n}} A_n}(z) = \frac{\partial}{\partial \Re(z)} \frac{2}{n} \log |\det( \frac{1}{\sqrt{n}} A_n - z I )|
\end{equation}
for all but finitely many $z$.

We have the following fundamental identity:

\begin{lemma}[Girko's identity]\label{girko-lemma}\cite{girko}  For every non-zero $u,v$ we have
$$ m_{\frac{1}{\sqrt{n}} A_n}(u,v) = \frac{u^2+v^2}{4\pi iu} \int_\R (\int_\R g_{\frac{1}{\sqrt{n}} A_n}(s+it) e^{i u s + i v t}\ dt) ds,$$
where the inner integral is absolutely integrable for almost every $s$, and the outer integral is absolutely convergent.
\end{lemma}

\begin{proof} We argue as in \cite[Lemma 3.1]{bai}.  Since
$$
m_{\frac{1}{\sqrt{n}} A_n}(u,v) = \frac{1}{n} \sum_{j=1}^n e^{i (u \Re( \frac{1}{\sqrt{n}} \lambda_j ) + v \Im( \frac{1}{\sqrt{n}} \lambda_j )) }$$
it suffices from \eqref{gobs} to show that
$$ e^{i(u \Re(w) + v \Im(w))} = \frac{u^2+v^2}{2\pi iu} \int_\R (\int_\R \frac{\Re(s+it-w)}{|s+it-w|^2} e^{i u s + i v t}\ dt) ds$$
for each complex number $w$, with an absolutely convergent inner integral and outer integral.  But standard contour integration shows that
\begin{equation}\label{contour}
\int_\R \frac{\Re(s+it-w)}{|s+it-w|^2} e^{i u s + i v t}\ dt = \pi \sgn(s-\Re(w)) e^{-v|s-\Re(w)|} e^{ius} e^{iv\Im(w)}
\end{equation}
for every $s \neq \Re(w)$, and the claim follows by an elementary integration.
\end{proof}

We can of course define $g_{\frac{1}{\sqrt{n}} B_n}$ similarly, with analogous identities.  To conclude the proof of Theorem \ref{theorem:replacement}, it thus suffices to show that for any $\eps > 0$ and any $n$, we have
\begin{equation}\label{grow}
\int_\R (\int_\R (g_{\frac{1}{\sqrt{n}} A_n}(s+it)-g_{\frac{1}{\sqrt{n}} B_n}(s+it)) e^{i u s + i v t}\ dt) ds = O(\eps)
\end{equation}
with probability $1-O(\eps)-o(1)$.

Fix $\eps > 0$.  By \eqref{tight-1}, \eqref{tight-2}, we can find an $R > 1$ large enough that with probability $1-O(\eps)$,
\begin{equation}\label{mob}
\mu_{\frac{1}{\sqrt{n}} A_n}(\{ z \in \BBC: |z| \geq R \}) +
\mu_{\frac{1}{\sqrt{n}} B_n}(\{ z \in \BBC: |z| \geq R \}) \leq \eps.
\end{equation}
We now condition on the event that \eqref{mob} holds.

We now smoothly localize the $z$ variable to a compact set as follows.
Let $\psi:\R \to \R^+$ be a smooth cutoff function which equals $1$ on $[-1,1]$ and is supported on $[-2,2]$.

\begin{lemma}[Truncation in $s, t$]\label{trunc}  Let $w \in \BBC$.
\begin{itemize}
\item[(i)] The integral
$$ \int_\R |\int_\R \frac{\Re(w-(s+it))}{|w-(s+it)|^2} e^{i u s + i v t}\ dt| (1 - \psi(s/R^2))\ ds$$
is of size $O(1)$, and (if $R$ is large enough) is of size $O(\eps)$ when $|w| \leq R$.
\item[(ii)] The integral
\begin{equation}\label{int}
 \int_\R |\int_\R \frac{\Re(w-(s+it))}{|w-(s+it)|^2} e^{i u s + i v t} (1 - \psi(t/R^2))\ dt| \psi(s/R^2)\ ds
 \end{equation}
is of size $O(1)$, and (if $R$ is large enough) is of size $O(\eps)$ when $|w| \leq R$.
\end{itemize}
\end{lemma}

\begin{proof}  The claim (i) follows easily from \eqref{contour}, so we turn to (ii).  We first verify the claim that \eqref{int} is bounded.  Replacing everything by absolute values one sees that
$$
|\int_\R \frac{\Re(w-(s+it))}{|w-(s+it)|^2} e^{i u s + i v t} (1 - \psi(t/R^2))\ dt| = O(1)$$
(in fact one can obtain an explicit upper bound of $\pi$), so we can dispose of the region of integration
in which $s = \Re(w)+O(1)$.  For the remaining values of $s$, we use repeated integration by parts,
integrating the $e^{ivt}$ term and differentiating the others.
After two such integrations we obtain the bound
$$|\int_\R \frac{\Re(w-(s+it))}{|w-(s+it)|^2} e^{i u s + i v t}
(1 - \psi(t/R^2))\ dt| = O( ( R^{-2} + |s-\Re(w)|^{-1} )^{2} ).$$
The claim then follows.

Finally, if $|w| \leq R$, then one easily verifies (by repeated integration by parts) that
$$
\int_\R \frac{\Re(w-(s+it))}{|w-(s+it)|^2} e^{i u s + i v t} (1 - \psi(t/R^2))\ dt = O(1/R^{4})$$
(say), and so the final claim of (ii) follows.
\end{proof}

From this lemma and \eqref{gobs}, the triangle inequality and \eqref{mob} we conclude that
\begin{equation}\label{j1}
\int_\R (\int_\R g_{\frac{1}{\sqrt{n}} A_n}(s+it) e^{i u s + i v t}\ dt) (1 - \psi(s/R^2)) ds = O(\eps).
\end{equation}
and
\begin{equation}\label{j2}
\int_\R (\int_\R g_{\frac{1}{\sqrt{n}} A_n}(s+it) e^{i u s + i v t} (1 - \psi(t/R^2)) \ dt) \psi(s/R^2) ds = O(\eps).
\end{equation}
From \eqref{j1}, \eqref{j2} (and their counterparts for $g_{\frac{1}{\sqrt{n}} B_n}$) and the triangle inequality, we thus see that to prove \eqref{grow}, it suffices to show that
\begin{equation}\label{joy}
\int_\R \int_\R (g_{\frac{1}{\sqrt{n}} A_n}(s+it)-g_{\frac{1}{\sqrt{n}} B_n}(s+it)) e^{i u s + i v t} \psi(t/R^2) \psi(s/R^2)\ dt ds
\end{equation}
converges in probability to zero for every fixed $R \geq 1$.  Note that the
integrands here are now jointly absolutely integrable in $t,s$, and
so we may now freely interchange the order of integration.

Fix $R$.  Using \eqref{snort} and integration by parts in the $s$ variable, we can rewrite \eqref{joy} in the form
$$
\int_\R \int_\R f_n(s,t) \phi_{u,v,R}(s,t)\ ds dt$$
where
$$ f_n(s,t) := \frac{1}{n} \log |\det( \frac{1}{\sqrt{n}} A_n - z I )| - \frac{1}{n} \log |\det( \frac{1}{\sqrt{n}} B_n - z I )|$$
and
$$ \phi_{u,v,R}(s,t) := - \frac{\partial}{\partial s} ( e^{i u s + i v t} \psi(t/R^2) \psi(s/R^2) ).$$
(Note that there are finitely many values of $t$ for which the integration by parts is not
justified due to singularities in $g_{\frac{1}{\sqrt{n}} A_n}$ or $g_{\frac{1}{\sqrt{n}} B_n}$, but these values of $t$ clearly give a zero contribution at the end of the day.)  Thus it will suffice to show that
$$
\int_\R \int_\R |f_n(s,t)| |\phi_{u,v,R}(s,t)|\ ds dt$$
converges in probability to zero.

From \eqref{det} we have
\begin{equation}\label{det2}
 \frac{1}{n} \log |\det( \frac{1}{\sqrt{n}} A_n - z I )| =
\frac{1}{n} \sum_{j=1}^n \log | \frac{1}{\sqrt{n}} \lambda_j - (s+it)|
\end{equation}
and similarly for $B_n$.  From the boundedness and compact support of $\phi_{u,v,R}$ we observe that
$$
\int_\R \int_\R \log| \frac{1}{\sqrt{n}} \lambda - (s+it)|^2 |\phi_{u,v,R}(s,t)|\ ds dt \leq O_{\phi_{u,v,R}}(1 + \frac{1}{n} |\lambda|^2)$$
for all $\lambda \in \C$; from this, \eqref{det2}, \eqref{lam}, and the triangle inequality we see that
\begin{equation}\label{fbound}
\int_\R \int_\R |f_n(s,t)|^2 |\phi_{u,v,R}(s,t)|\ ds dt
\end{equation}
is bounded
uniformly in $n$.  Since by hypothesis $f_n(s,t)$ converges in probability to zero
for almost every $s,t$, the claim now follows from dominated
convergence (Lemma \ref{dom}). The proof of Theorem
\ref{theorem:replacement} is now complete in the case of convergence in probability.

\subsection{The almost sure convergence case}  We now indicate how to adapt the above arguments to the case of almost sure convergence.  Firstly, since \eqref{pan} is now almost surely bounded instead of just bounded in probability, we can now say that for every $\eps > 0$ there exists $C_\eps > 0$ such that with probability $1-O(\eps)$, \eqref{tight-1}, \eqref{tight-2} holds for all sufficiently large $n$ (as opposed to these bounds holding with probability $1-O(\eps)-o(1)$ for each $n$ separately).

Next, we observe the (well-known) fact that Lemma \ref{probconv} continues to hold when convergence in probability is replaced by almost sure convergence throughout.  Indeed the implication of (ii) from (i) is nearly identical and is left as an exercise to the reader.  To deduce (i) from (ii) in the almost sure case, observe from the separability of the space of smooth compactly supported functions in the uniform topology that it suffices to show that \eqref{mufan} converges almost surely to zero for each $f$.  On the other hand, from (ii) and Fubini's theorem we know that with probability $1$, that $m_{\frac{1}{\sqrt{n}} A_n}(u,v) - m(u,v)$ converges to zero for almost every $u,v$, and the claim follows from the (ordinary) dominated convergence theorem.

Once again we use Girko's identity, Lemma \ref{girko-lemma}, and reduce to showing that for every $\eps > 0$, one has with probability $1-O(\eps)$ that \eqref{grow} holds for all but finitely many $n$.  From our bounds on \eqref{tight-1}, \eqref{tight-2} we see that with probability $1-O(\eps)$, that \eqref{mob} holds for all but finitely many $n$.  We apply Lemma \ref{trunc} (which is deterministic) and reduce to showing that \eqref{joy} converges almost surely to zero for each fixed $R \geq 1$.  The rest of the argument proceeds as in the convergence in probability case.

\subsection{An alternate argument}

There is an alternate derivation\footnote{We thank Manjunath Krishnapur for this simpler argument.} of Theorem \ref{theorem:replacement} that avoids Fourier analysis, and is instead based on the observation that for any complex polynomial $P(z)$, the distributional Laplacian $\Delta \log |P(z)|$ of the logarithm of the magnitude of $P$ is equal to the counting measure of the zeroes of $P$ (counting multiplicity).  In particular, we see from Green's theorem that
$$\int_{\C} f\ d(\mu_{\frac{1}{\sqrt{n}} A_n} - \mu_{\frac{1}{\sqrt{n}} B_n}) = 
\frac{1}{2\pi n} \int_\C (\Delta f(z)) \log |\det(\frac{1}{\sqrt{n}} A_n - zI)| - \frac{1}{n} \log |\det(\frac{1}{\sqrt{n}} B_n -
 zI)|$$
for any smooth, compactly supported $f$.  Applying Lemma \ref{dom} we can then get convergence of this integral (either in probability or in the almost sure sense, as appropriate); the uniform integrability required can be established by repeating the computations used to bound \eqref{fbound}.  One can then easily take limits to replace smooth compactly supported $f$ to continuous compactly supported $f$; we omit the details.

\section{Proof of Proposition \ref{lemma:determinant}} \label{detsec}

In this section we present  the proof of Proposition
\ref{lemma:determinant}, modulo several key lemmas. Let
$\a,\b,M_n,A_n,B_n,z$ be as in that proposition. By shifting $M_n$
by $\sqrt{n} z I$ if necessary we can assume $z=0$.
Our task is now
to show that
$$\frac{1}{n} \log |\det(\frac{1}{\sqrt{n}} A_n)| - \frac{1}{n} \log |\det(\frac{1}{\sqrt{n}} B_n)|$$
converges in probability to zero, and also almost surely to zero if $\mu_{\frac{1}{n} M_n M_n^\ast}$ converges.

Let us first remark that the almost sure convergence claim implies the convergence in probability claim.  Indeed, suppose that convergence in probability failed, then there would exist an $\eps > 0$ such that
\begin{equation}\label{pbn}
 \P\left( \left|\frac{1}{n} \log |\det(\frac{1}{\sqrt{n}} A_n)| - \frac{1}{n} \log |\det(\frac{1}{\sqrt{n}} B_n)|\right| \geq \eps \right) \geq \eps
 \end{equation}
for a subsequence of $n$.  By vague sequential compactness one can pass to a further subsequence along which $\mu_{\frac{1}{n} M_n M_n^\ast}$ converges, and hence by hypothesis one has almost sure (and hence in probability) convergence to zero along this sequence, contradicting \eqref{pbn}.  Thus it suffices to establish almost sure convergence assuming the convergence of $\mu_{\frac{1}{n} M_n M_n^\ast}$.

Let $Z_1,\ldots,Z_n$ be the rows of $M_n$.  By assumption
\eqref{eqn:conditionM} we
have
$$ \sum_{j=1}^n \|Z_i\|^2 = O(n^2).$$
In particular, at least half of the $Z_i$ have norm $O(\sqrt{n})$.
 By permuting the rows of $M_n, A_n, B_n$ if necessary, we may assume that
  it the last half of the rows have this property, thus
\begin{equation}\label{zhalf}
\|Z_i\| = O(\sqrt{n}) \hbox{ for all } n/2 \leq i \leq n.
\end{equation}

Let $\sigma_1(A) \geq \ldots \geq \sigma_n(A) \geq 0$ denote the singular values of a matrix $A$.  We have the following fundamental lower bound:

\begin{lemma}[Least singular value bound]  \label{lemma:lsv} With probability $1$, we have
\begin{equation}\label{ans}
\sigma_n(A_n), \sigma_n(B_n)  \geq n^{-O(1)}
\end{equation}
for all but finitely many $n$. In particular, with probability $1$, $A_n$ and $B_n$ are invertible for all but finitely many $n$.
\end{lemma}

\begin{proof} This follows immediately from \cite[Theorem 2.1]{TV-circular}
or \cite[Theorem 4.1]{TV-lsv} and the Borel-Cantelli lemma, noting from
\eqref{eqn:conditionM} of Proposition \ref{lemma:determinant}
that the operator norm of $M_n$ is of polynomial size $n^{O(1)}$.
 There are  previous results in \cite{Rud},
 \cite{TVsing}, \cite{RV}, \cite{TVstoc},
 which handled special cases with more assumptions on $M_n$ and
 the underlying distributions $\a, \b$ (for instance, in some of the prior results $M_n$
 was assumed to vanish, or $\a, \b$ were assumed to be integer-valued or to have finite higher moments).
 One can obtain explicit bounds on the tail probability and on the exponent $O(1)$; see \cite{TV-lsv}.
 However, for our applications the above bounds will suffice.
\end{proof}

We also have with probability $1$ the crude upper bound
\begin{equation}\label{bns}
\sigma_1(A_n), \sigma_1(B_n) \leq n^{O(1)}
\end{equation}
for all but finitely many $n$, which follows easily from the polynomial size of $M_n$ the
bounded second moment of $\a, \b$, and the Borel-Cantelli lemma.
Again, much sharper bounds are available, especially if $\a$ and $\b$ have
 finite fourth moment, but we will not need these bounds here.

Let $X_1,\ldots,X_n$ be the rows of $A_n$, and for each $1 \leq i \leq n$ let $V_i$ be the $i-1$-dimensional space generated by $X_1,\ldots,X_{i-1}$.  From \eqref{det} we have
$$ \frac{1}{n} \log |\det(\frac{1}{\sqrt{n}} A_n)| =
\frac{1}{n} \sum_{i=1}^n \log \dist( \frac{1}{\sqrt{n}} X_i, V_i )$$
and similarly
$$ \frac{1}{n} \log |\det(\frac{1}{\sqrt{n}} B_n)| =
\frac{1}{n} \sum_{i=1}^n \log \dist( \frac{1}{\sqrt{n}} Y_i, W_i )$$
where $Y_1,\ldots,Y_n$ are the rows of $\frac{1}{\sqrt{n}} B_n$, and $W_i$ is spanned by $Y_1,\ldots,Y_{i-1}$.  Our task is then to show that
$$ \frac{1}{n} \sum_{i=1}^n \log \dist( \frac{1}{\sqrt{n}} X_i, V_i ) - \log \dist( \frac{1}{\sqrt{n}} Y_i, W_i )$$
converges almost surely to zero.

From \eqref{ans}, \eqref{bns} and Lemma \ref{lemma:twosum} we almost surely obtain the bound
$$ \log \dist( \frac{1}{\sqrt{n}} X_i, V_i ), \log \dist( \frac{1}{\sqrt{n}} Y_i, W_i ) = O( \log n )$$
for all but finitely many $n$. Thus it suffices to show that
$$ \frac{1}{n} \sum_{1 \leq i \leq n-n^{0.99}} \log \dist( \frac{1}{\sqrt{n}} X_i, V_i ) -
\log \dist( \frac{1}{\sqrt{n}} Y_i, W_i )$$
(say) converges almost surely to zero.
This follows immediately from the following two lemmas.

\begin{lemma}[High-dimensional contribution]\label{highdim}  For every $\eps > 0$ there exists $0 < \delta < 1/2$ such that with probability $1$, one has
$$ \frac{1}{n} \sum_{(1-\delta)n \leq i \leq n-n^{0.99}} |\log \dist( \frac{1}{\sqrt{n}} X_i, V_i )| = O(\eps)$$
for all but finitely many $n$.
Similarly with $\dist(\frac{1}{\sqrt{n}} X_i,V_i)$ replaced by $\dist(\frac{1}{\sqrt{n}} Y_i,W_i)$.
\end{lemma}

\begin{lemma}[Low-dimensional contribution]\label{lowdim}  For every $\eps > 0$ there exists $0 < \delta < 1/2$, such that with probability $1-O(\eps)$, one has
$$ \frac{1}{n} \sum_{1 \leq i \leq (1-\delta) n} \log \dist( \frac{1}{\sqrt{n}} X_i, V_i ) - \log \dist( \frac{1}{\sqrt{n}} Y_i, W_i ) = O(\eps)$$
for all but finitely many $n$.
\end{lemma}

The next two sections will be devoted to the proofs of these two lemmas.

\section{Proof of Lemma \ref{highdim}} \label{section:highdim}

We now prove Lemma \ref{highdim}.  We can of course take $n$ to be large depending on all fixed parameters.
 Let $0 < \delta < 1/2$ be a small number depending on $\eps$ to be chosen later.

Clearly it suffices to prove this lemma for $\dist(\frac{1}{\sqrt{n}} X_i,V_i)$.
We first prove the (much easier) bound for the positive component of the logarithm.  By the Borel-Cantelli lemma it suffices to show that
$$ \sum_{n=1}^\infty \P( \frac{1}{n} \sum_{(1-\delta)n \leq i \leq n-n^{0.99}}
 \max(\log \dist( \frac{1}{\sqrt{n}} X_i, V_i ),0) \geq \eps ) < \infty.$$
To establish this, we use the crude bound
$$ \max(\log \dist( \frac{1}{\sqrt{n}} X_i, V_i ),0)
\leq \max(\log \frac{1}{\sqrt{n}} \|X_i\|,0) $$
and thus
\begin{equation}\label{pig}
 \frac{1}{n} \sum_{(1-\delta)n \leq i \leq n-n^{0.99}}
 \max(\log \dist( \frac{1}{\sqrt{n}} X_i, V_i ),0)
 \leq O( \sum_{m=0}^\infty \frac{1}{n} \sum_{(1-\delta)n \leq i \leq n-n^{0.99}} \I( \|X_i\| \geq 2^m \sqrt{n} ) ).
 \end{equation}
 Thus if the left-hand side of \eqref{pig} exceeds $\eps$, we must have
$$ \frac{1}{n} \sum_{(1-\delta)n \leq i \leq n-n^{0.99}} \I( \|X_i\| \geq 2^m \sqrt{n} ) \geq \eps/(100+m)^2$$
(say) for some $m \geq 0$.  On the other hand, from \eqref{zhalf} and the second moment method we see that $\P( \|X_i\| \geq 2^m \sqrt{n} ) = O( 2^{-2m} )$, and thus by Hoeffding's inequality we have
$$ \P( \frac{1}{n} \sum_{(1-\delta)n \leq i \leq n-n^{0.99}} \I( \|X_i\| \geq 2^m \sqrt{n} ) \geq \eps/(100+m)^2 ) \leq C \exp( - c n^{-0.01} - c m^{-0.01} )$$
(say) for some constants $C, c > 0$ depending on $\eps$, if $\delta$ is chosen sufficiently small depending on $\eps$.  The claim follows.

It remains to establish the bound for the negative component of the logarithm.  By the Borel-Cantelli lemma it suffices to show that
$$ \sum_{n=1}^\infty \P( \frac{1}{n} \sum_{(1-\delta)n \leq i \leq n-n^{0.99}} \max(-\log \dist( \frac{1}{\sqrt{n}} X_i, V_i ),0) \geq \eps ) < \infty.$$
This will follow from the union bound and
the following estimate.

\begin{proposition}[Lower tail bound]\label{xw}  Let $1 \leq d \leq n-n^{0.99}$ and $0 < c < 1$,
and let $W$ be a (deterministic) $d$-dimensional subspace of $\BBC^n$.  Let $X$ be a row of $A_n$ (the exact choice of row is not important).  Then
$$ \P( \dist( X, W ) \leq c \sqrt{n-d} ) = O( \exp( - n^{0.01} ) ).$$
(The implied constant of course depends on $c$.)
\end{proposition}

Indeed, since $X_i$ and $V_i$ are independent of each other, the proposition implies that
$$ \dist( \frac{1}{\sqrt{n}} X_i, V_i ) \geq \frac{1}{2\sqrt{n}} \sqrt{n-i+1}$$
(say) for each $(1-\delta)n \leq i \leq n-n^{0.99}$, with probability
$1-O( n^{-10} )$ (say). Setting $\delta$ sufficiently small (compared to
$\epsilon)$, taking logarithms and summing in $i$ and $n$ one obtains the
claim.

It remains to prove the proposition.  Similar lower bounds
concerning the distance of a random vector to a fixed subspace
have appeared in \cite{TV1}, \cite{RV},
\cite{RV2}.  Here, however, we have
the complication that the coefficients of $X$ have non-zero mean and have no higher moment bounds
than the second moment; in particular, they can be unbounded.

We first eliminate the problem that $X$ has non-zero mean.
 Write $X = v + X'$, where $v := \E(X)$ is a deterministic
 vector (which could be quite large) and $X'$ has mean zero.
  Then we have $\dist(X,W) \geq \dist(X', \Span(W,v))$.
  Thus Proposition \ref{xw} follows from the mean zero case
  (after making the harmless change of incrementing $d$ to $d+1$,
  and adjusting the parameters slightly to suit this).

Henceforth we assume that $X$ has mean zero, thus $X = (\a_1,\ldots,\a_n)$ for some iid
copies $\a_1,\ldots,\a_n$ of $\a$.  Now we deal with the problem that the $\a_1,\ldots,\a_n$ can be unbounded.
 By Chebyshev's inequality, we have $\P( |\a_i| \geq n^{0.1} ) = O( n^{-0.2} )$
 for all $1 \leq i \leq n$.  The event $|\a_i| \ge n^{0.1}$
  are jointly independent in $i$.
  By  Chernoff inequality (see, for instance, \cite[Chapter 1]{TVbook}),
  we can show  that with probability $1 - O( \exp( - n^{0.01} ) )$,
  that there are  at most $n^{0.9}$ indices $i$ for which $|\a_i| \geq n^{0.1}$.
  (One can also verify this directly using binomial coefficients and
  Sterling's formula.)

By conditioning on the various possible sets of indices for which $|\a_i| \geq n^{0.1}$,
we see that it suffices to show that
$$ \P( \dist( X, W) \leq c \sqrt{n-d}  | E_I)
= O( \exp( - n^{0.01} ) )$$
for each $I \subset \{1,\ldots,n\}$ of cardinality at most $n^{0.9}$, where $E_I$ is the event
that $I = \{ 1 \leq i \leq n: |\a_i| \geq n^{0.1} \}$.


Without loss of generality we can take $I = \{ n'+1,\ldots,n\}$ for some $n-n^{0.9} \leq n' \leq n$.
We then observe that
$$ \dist(X,W) \geq \dist(\pi(X),\pi(W))$$
where $\pi: \BBC^n \to \BBC^{n'}$ is the orthogonal projection.  By
conditioning on the coordinates $\a_{n'+1},\ldots,\a_n$ and making
the minor change of replacing $n$ with $n'$ (and adjusting $c$
slightly), we may thus reduce to the case when $I$ is empty, thus it
suffices to show that
$$ \P( \dist( X, W) \leq c \sqrt{n-d} | |\a_i| < n^{0.1} \hbox{ for all } i ) = O( \exp( - n^{0.01} ) ).$$
Let $\tilde \a$ be the random variable $\a$ conditioned to the event $|\a| < n^{0.1}$, and let $\tilde X = (\tilde \a_1,\ldots,\tilde \a_n)$ be a vector consisting of iid copies of $\tilde \a$.  It then suffices to show that
\begin{equation}\label{distw}
 \P( \dist( \tilde X, W) \leq c \sqrt{n-d} ) = O( \exp( - n^{0.01} ) ).
\end{equation}

Note that $\tilde \a$ might have a non-zero mean, but this can be easily dealt
with by the same trick used before, subtracting $\E \tilde \a$ from $\tilde \a$ to make $X$ to have zero mean.  Since $\a$ had variance $1$, we see from monotone convergence that $\tilde \a$ has variance $1-o(1)$.

To prove \eqref{distw}, we recall the following inequality of Talagrand.

\begin{theorem}[Talagrand's inequality]\label{theorem:Talagrand}
Let $\D$ be the unit disk  $\{z\in \BBC, |z| \le 1 \}$. For every
product probability $\mu$ on $\D^n$, every convex $1$-Lipschitz
function $F: \BBC^n \to \R$, and every $r \ge 0$,
$$\mu (|F- M(F)| \ge r) \le 4 \exp(-r^2/8), $$
where $M(F)$ denotes the median of $F$.
\end{theorem}

\begin{proof} This is the complex version of \cite[Corollary 4.10]{Le}, in which
$\D$ was replaced by the unit interval $[0,1]$. The proof is the
same, with a slight modification that implies a worse the constant
($1/8$ instead of $1/4$) in the exponent.
\end{proof}

We apply this theorem with $\mu$ equal to the distribution of
$\tilde X / n^{0.1}$ and $F: \BBC^n \to \R$ equal to the convex
$1$-Lipschitz function $F(v) := \dist(v,W)$, and conclude that
\begin{equation}\label{median}
 \P( |\dist( \tilde X, W) - M(\dist( \tilde X, W))| \geq n^{0.1} r ) \leq 4 \exp(-r^2/8)
 \end{equation}
for every $r > 0$.  On the other hand, we can easily compute the second moment (cf. \cite[Lemma 2.5]{TV1}):

\begin{lemma}\label{lemma:secondmomentofdist} We have
$$ \E( \dist( \tilde X, W)^2 ) =(1-o(1))( n-d).$$
\end{lemma}

\begin{proof}  Let $\pi = (\pi_{ij})_{1 \leq i,j \leq n}$ be the orthogonal projection matrix to $W$.  Observe that $\dist(\tilde X,W)^2 = \sum_{i=1}^n \sum_{j=1}^n \tilde \a_i \pi_{ij} \overline{\tilde \a_j}$.  Since the $\tilde \a_i$ are iid with mean zero, we thus have
$$ \E( \dist( \tilde X, W)^2 ) = (\E \tilde \a^2) \sum_{i=1}^n \pi_{ii}.$$
But $\sum_{i=1}^n \pi_{ii} = \trace(\pi)$ is equal to $\tilde n$.  Since $\tilde \a$ had variance $1-o(1)$, the claim follows.
\end{proof}

Since $n-d \geq n^{0.99}$ and $c < 1$, the claim
\eqref{distw} from follows from \eqref{median} and the above lemma. The proof of Lemma \ref{highdim} is now
complete.

\section{Proof of Lemma \ref{lowdim}} \label{section:lowdim}

We now begin the proof of Lemma \ref{lowdim}.  Fix $\eps$, and assume that $\delta$ is sufficiently small depending on $\eps$.  Write $n' :=
\lfloor (1-\delta) n \rfloor$.  Observe that $\prod_{i=1}^{n'}
\dist( \frac{1}{\sqrt{n}} X_i, V_i )$ is the $n'$-dimensional volume
of the parallelepiped spanned by $X_1,\ldots,X_{n'}$, which is also
equal to $\det(\frac{1}{n} A_{n,n'} A_{n,n'}^\ast)^{1/2}$, where
$A_{n,n'}$ is the $n' \times n$ matrix with rows
$X_1,\ldots,X_{n'}$.  Expressing this determinant as the product of
singular values, we conclude the identity
$$ \frac{1}{n} \sum_{1 \leq i \leq (1-\delta) n} \log \dist( \frac{1}{\sqrt{n}} X_i, V_i ) =
\frac{1}{n} \sum_{i=1}^{n'} \log\left(\frac{1}{\sqrt{n}} \sigma_i(A_{n,n'})\right).$$
Similarly for $Y_i, W_i$, and $B_{n,n'}$ (the matrix generated by $Y_1,\ldots,Y_{n'}$.  Thus it suffices to show that with probability $1-O(\eps)$, one has
\begin{equation}\label{lnu}
\frac{1}{n'} \sum_{i=1}^{n'} \log\left(\frac{1}{\sqrt{n}} \sigma_i(A_{n,n'})\right) - \log \left(\frac{1}{\sqrt{n}} \sigma_i(B_{n,n'})\right) = O(\eps)
\end{equation}
for all but finitely many $n$.
We rewrite \eqref{lnu} as
\begin{equation}\label{lnu2}
\int_0^\infty \log t\ d\nu_{n,n'}(t) = O(\eps)
\end{equation}
where $d\nu_{n,n'}$ is the difference of two ESDs:
$$ d\nu_{n,n'} = \mu_{\frac{1}{n'} A_{n,n'} A_{n,n'}^\ast} - \mu_{\frac{1}{n'} B_{n,n'} B_{n,n'}^\ast}.$$

We control \eqref{lnu} by dividing the range of $t$ into several parts.

\subsection{The region of very large $t$}\label{section:larget}

We now control the region where $t \geq R_\eps$ for some large $R_\eps$.

From Lemma \ref{compar} we have that
$$ \frac{1}{n} \sum_{i=1}^{n'} (\frac{1}{\sqrt{n}} \sigma_i(A_{n,n'}))^{2},
\frac{1}{n} \sum_{i=1}^{n'} (\frac{1}{\sqrt{n}} \sigma_i(B_{n,n'}))^{2}$$
is almost surely bounded, and thus
$$ \int_0^\infty t |d\nu_{n,n'}(t)|$$
is also almost surely bounded.  Thus, with probability $1-O(\eps)$, we have
$$ \int_0^\infty t |d\nu_{n,n'}(t)| \leq C_\eps$$
for all but finitely many $n$, and some $C_\eps$ independent of $n$, which implies that
\begin{equation}\label{fire-1}
\int_{R_\eps}^\infty |\log t| |d\nu_{n,n'}(t)| \leq \eps
\end{equation}
for all but finitely many $n$, and some $R_\eps$ depending only on $\eps$.

\subsection{The region of intermediate $t$}\label{section:intermediatet}

We now control the region $\eps^4 \leq t \leq R_\eps$.

\begin{lemma}\label{lemma:intermediatet}  Let $\psi$ be a smooth function which equals $1$ on $[\eps^4,R_\eps]$ and is supported on $[\eps^4/2,2R_\eps]$.  Then with probability $1$, we have
\begin{equation}\label{fire-2}
\int_0^\infty \psi(t) \log t d\nu_{n,n'}(t) = O(\eps),
\end{equation}
if $\delta$ is sufficiently small depending on $\eps$ and $\psi$.
\end{lemma}

\begin{proof}  From the interlacing property (Lemma \ref{lemma:interlacing}), we see that
$$ \int_0^\infty \psi(t) \log t d\nu_{n,n'}(t) =
\int_0^\infty \psi(t) \log t d\nu_{n,n}(t) + O(\eps)
$$
if $\delta$ is sufficiently small depending on $\eps$ and $\psi$.

We now apply the recent result in \cite[Theorem 1.1]{doz}. For the
reader's convenience, we restate this result in the Appendix; see
Theorem \ref{theorem:DS}. This result asserts under the above hypotheses  that the ESDs $d\mu_{\frac{1}{n} A_{n} A_{n}^\ast}$ and $d\mu_{\frac{1}{n} B_{n} B_{n}^\ast}$  converge almost surely to the same limit (in fact, this limit is given explicitly in terms of the limiting distribution of $\mu_{\frac{1}{n} M_n M_n^\ast}$ via the inverse Stieltjes transform of \eqref{eqn:DS}).  In particular, $\nu_{n,n}$ converges almost surely to zero, and the claim follows.
\end{proof}

\begin{remark}\label{detsec-remark} Note that for the convergence in probability case of Proposition \ref{lemma:determinant}, we need to apply Theorem \ref{theorem:DS} to a subsequence of $n$ rather than to all $n$, thanks to the subsequence extraction performed at the beginning of Section \ref{detsec}.
\end{remark}

\subsection{The region of moderately small $t$}\label{section:moderatet}

We now control the region $\delta^2 \leq t \leq \eps^4$.  For this we need some bounds on the low singular values of $A_{n,n'}$ and $B_{n,n'}$.

\begin{lemma}\label{ab1}  With probability $1$, we have
\begin{equation}\label{san}
\frac{1}{n} \sum_{i=1}^{n'} (\frac{1}{\sqrt{n}} \sigma_i(A_{n,n'}))^{-2} = O(1)
\end{equation}
for all but finitely many $n$, and similarly with $A_{n,n'}$ replaced by $B_{n,n'}$.
\end{lemma}

\begin{proof}
Clearly it suffices to establish the claim for $A_{n,n'}$.
Using Proposition \ref{xw} and the Borel-Cantelli lemma, we see that with probability $1$, we have
$$ \dist( \frac{1}{\sqrt{n}} X_i, \Span(X_1,\ldots,X_{i-1},X_{i+1},\ldots,X_{n'}) ) \geq \frac{1}{2} \sqrt{\delta n}$$
for all but finitely many $n$, and all $1 \leq i \leq n'$.  The claim then follows from Lemma \ref{lemma:twosum}.
\end{proof}

Since the $\sigma_i(A_{n,n'})$ are decreasing in $i$, and $n' = \lfloor (1-\delta) n \rfloor$, we see that the above lemma implies that with probability $1$, we have
$$ \frac{1}{\sqrt{n}} \sigma_{\lfloor(1-2\delta) n \rfloor}(A_{n,n'}) \geq c \delta$$
for all but finitely many $n$, and some absolute constant $c > 0$.  We can generalize this lower bound to handle higher singular values also:

\begin{lemma}\label{slib}  There exists an absolute constant $c > 0$ such that with probability $1$, we have
\begin{equation}\label{san2}
 \frac{1}{\sqrt{n}} \sigma_i(A_{n,n'}) \geq c \frac{n'-i}{n}
\end{equation}
for all but finitely many $n$, and all $1 \leq i \leq (1-2\delta) n$, and similarly with $A_{n,n'}$ replaced by $B_{n,n'}$.
\end{lemma}

\begin{proof}  Clearly it suffices to establish the claim for $A_{n,n'}$.
Using Proposition \ref{xw} and the Borel-Cantelli lemma, we see that with probability $1$, we have
$$ \dist( \frac{1}{\sqrt{n}} X_i, \Span(X_1,\ldots,X_{i-1},X_{i+1},\ldots,X_{n''}) ) \geq \frac{1}{2} \sqrt{n-n''}$$
for all but finitely many $n$, and all $1 \leq i \leq n''$ and $n/2 \leq n'' \leq n'$.
Applying Lemma \ref{lemma:twosum}, we conclude that we almost surely have
$$ \frac{1}{n} \sum_{i=1}^{n''} (\frac{1}{\sqrt{n}} \sigma_i(A_{n,n''}))^{-2} = O(\frac{n}{n-n''})$$
for all but finitely many $n$, and all $n/2 \leq n'' \leq n'$.  Using the crude bound
$$ \sum_{i=1}^{n''} (\frac{1}{\sqrt{n}} \sigma_i(A_{n,n''}))^{-2} \geq (n-n'')
(\frac{1}{\sqrt{n}} \sigma_{2n''-n}(A_{n,n''}))^{-2}$$
we conclude that we almost surely have
$$ \frac{1}{\sqrt{n}} \sigma_{2n''-n}(A_{n,n''}) \geq c' \frac{n-n''}{n}$$
for all but finitely many $n$, all $n/2 \leq n'' \leq n'$, and some absolute constant $c' > 0$.  The claim now follows from the Cauchy interlacing property (Lemma \ref{lemma:interlacing}).
\end{proof}

\begin{remark} If one assumes stronger moment assumptions (e.g subgaussian) on $\a$, then more precise bounds are known, especially in the $M_n=0$ case: see \cite{RV2}, \cite{RV3}.
\end{remark}

From this lemma we can now bound the relevant contribution to \eqref{lnu}:

\begin{lemma}\label{go} With probability $1$, and if $\delta$ is sufficiently small depending on $\eps$, we have
\begin{equation}\label{fire-3}
\int_{\delta^2}^{\eps^4} |\log t| |d\nu_{n,n'}(t)| = O(\eps)
\end{equation}
for all but finitely many $n$.
\end{lemma}

\begin{proof} By the triangle inequality and symmetry it suffices to show that with probability $1$, we have
$$ \int_{\delta^2}^{\eps^4} |\log t| d\mu_{\frac{1}{n'} A_{n,n'} A_{n,n'}^\ast}(t) = O(\eps) $$
for all but finitely many $n$.  We rewrite the left-hand side as
$$ \frac{1}{n} \sum_{i=1}^{n'} f( \frac{1}{\sqrt{n}} \sigma_i(A_{n,n'}) )$$
where $f(t) := |\log t| \I( \delta^2 \leq t^2 \leq \eps^4 )$.  Since $f$ cannot exceed $|\log \delta|$, we see that the contribution of the case $i \geq (1-2\delta) n$ is acceptable if $\delta$ is small enough, so it suffices to show that we almost surely have
$$ \frac{1}{n} \sum_{1 \leq i \leq (1-2\delta) n} f( \frac{1}{\sqrt{n}} \sigma_i(A_{n,n'}) ) = O(\eps)$$
for all but finitely many $n$.

By Lemma \ref{slib}, we may assume that $n$ is such that \eqref{san2} holds.  As a consequence, we see that the only terms in the above sum which are non-vanishing are those for which $i = (1 - O(\eps^2)) n$.  But then if we apply \eqref{san2} and crudely estimate $f(t) \leq -\log t$ we obtain the claim.
\end{proof}

\subsection{The contribution of very small $t$}\label{section:smallt}

Finally, we need to control the contribution when $t \leq \delta$.

\begin{lemma} With probability $1$, and if $\delta$ is sufficiently small depending on $\eps$, we have
\begin{equation}\label{fire-4}
\int_0^{\delta^2} |\log t| |d\nu_{n,n'}(t)| = O(\eps)
\end{equation}
for all but finitely many $n$.
\end{lemma}

\begin{proof} By arguing as in the proof of Lemma \ref{go}, it suffices to show that we almost surely have
$$ \frac{1}{n} \sum_{i=1}^{n'} g( \frac{1}{\sqrt{n}} \sigma_i(A_{n,n'}) ) = O(\eps)$$
for all but finitely many $n$, where $g(t) := |\log t| \I( t^2 \leq \delta^2 )$.

By Lemmas \ref{ab1}, we may assume $n$ is such that \eqref{san} holds.  On the other hand, if $\delta$ is small enough, we have the bound $g(t) \leq \eps t^{-2}$.  The claim now follows from \eqref{san}.
\end{proof}

Putting together \eqref{fire-1}, \eqref{fire-2}, \eqref{fire-3}, \eqref{fire-4} we see that with probability $1-O(\eps)$, we have \eqref{lnu2} for all but finitely many $n$, and the claim follows.

\section{Extensions} \label{section:extensions}

\subsection{Proof of Theorem \ref{theorem:main2}}

The theorem in the case of almost sure convergence follows immediately from Theorem \ref{theorem:main1} by conditioning on $M_n$, so it remains to verify the theorem in the case of convergence in probability.

Let fix a test function $f$ (as in \eqref{fuz}) and a positive $\eps$. By the boundedness in probability of $\frac{1}{n^2} \|M\|_2^2$, we can find a $C = C_\eps$ such that $\P( M_n \in \Omega_n ) \geq 1-\eps$, where
$$ \Omega_n := \{ M \in M_n(\C): \frac{1}{n^2} \|M\|_2^2 \leq C \}.$$
Let $M_n^f$ be the matrix in $\Omega_n$
which maximizes\footnote{If the maximum is not attained, one can instead choose $M_n^f$ to be a matrix which maximizes this quantity to within a factor of two (say).} the quantity
$$\P( |\int_\BBC f(z)\ d\mu_{\frac{1}{\sqrt{n}} (M_n^f +X_n}(z) - \int_\BBC f(z)\ d\mu_{\frac{1}{\sqrt{n}} (M_n^f +Y_n}(z)|
\ge \eps
).$$
Applying Theorem \ref{theorem:main1} to the sequence $M_n^f +X_n$
and $M_n^f + Y_n$, we see that this quantity is $o(1)$.

Theorem \ref{theorem:main2} follows by integrating over all possible values of $M_n$ using the definition of $M_n^f$, as well as the fact that $\P(\Omega_n) \geq 1-\eps$, and then letting $\eps \to 0$.

\subsection{Proof of Theorem \ref{theorem:main3}} We first verify the claim for convergence in probability.

The condition (i) of Theorem \ref{theorem:replacement} is satisfied
thanks to the boundedness in probability of \eqref{eqn:trace-3}.  In order to complete the proof, one needs
to check (ii). Notice that

$$\det (\frac{1}{\sqrt n} A_n -zI) = \det (\frac{1}{\sqrt n}
(K_n^{-1} M_n L_n^{-1} + X_n ) - z K_n^{-1} L_n^{-1} ) \det L_n K_n.
$$

The term  $\det L_n K_n$ also appears in $\det (\frac{1}{\sqrt n}
B_n -zI)$ and becomes additive (and thus cancels) after taking
logarithm. Therefore, one only needs to show that

\begin{eqnarray*} & \frac{1}{n} \log |\det \Big( \frac{1}{\sqrt n}(K_n^{-1} M_n L_n^{-1}
+ X_n ) - z K_n^{-1} L_n^{-1} \Big)| \\ &- \frac{1}{n} \log |\det
\Big( \frac{1}{\sqrt n}(K_n^{-1} M_n L_n^{-1} + Y_n ) - z K_n^{-1}
L_n^{-1}\Big)| \end{eqnarray*} converges in probability to zero.

One can obtain this by repeating the proof of Proposition
\ref{lemma:determinant}. The slight change here is that $zI$ is
replaced by $z K_n^{-1} L_n^{-1}$, but this has no significant
impact, except that we need to show

$$F_n: = \frac{1}{\sqrt n}(K_n^{-1} M_n L_n^{-1}- z K_n^{-1}
L_n^{-1}) $$ satisfies

$$\frac{1}{n^2} \trace F_nF_n^\ast = \frac{1}{n^2} \|F_n\|_2^2 =O(1)$$
almost surely (in order to guarantee \eqref{eqn:conditionM}).
But this is a consequence of the boundedness in probability of \eqref{eqn:trace-3}.

The proof of the almost sure convergence is established similarly, with the obvious changes (e.g. replacing boundedness in probability with almost sure boundedness).  We omit the details.

\section{Proof of Theorem \ref{theorem:main4}}\label{thm4}

We first prove that (ii) implies (i) for almost sure convergence.
Let $A_n$ and $\mu$ be as in Theorem \ref{theorem:main4}. Construct a diagonal matrix
$B'_n$ whose diagonal entries are independent samples from $\mu$ and
let $B_n := \sqrt n B'_n $. We wish to invoke Theorem \ref{theorem:replacement}.  We first need to verify the almost sure boundedness of \eqref{pan}.  The bound for $A_n$ follows from Lemma \ref{tight}, and the bound for $B_n$ follows from the second moment hypothesis on $\mu$ and the (strong) law of large numbers. By Theorem \ref{theorem:replacement}, the problem now reduces to showing that for almost all complex numbers $z$,
$$\frac{1}{n}
 \log |\det(\frac{1}{\sqrt{n}} A_n - zI)| - \frac{1}{n} \log |\det(\frac{1}{\sqrt{n}} B_n -
 zI)|$$
 converges almost surely to zero.
The right hand side is easy to compute:
$$\frac{1}{n} \log |\det(\frac{1}{\sqrt{n}} B_n -
 zI)| =\frac{1}{n} \log |\det (B'_n -zI)| = \frac{\sum_{i=1}^n \log
 |\lambda_i-z| }{n},$$
 \noindent where $\lambda_i$ are iid samples from $\mu$. On the other hand, from Fubini's theorem we see that
$\int_{\BBC} \log|w-z|\ d\mu(w)$ is locally integrable in $z$, and thus
\begin{equation}\label{local}
\int_{\BBC} \log|w-z|\ d\mu(w) < \infty
\end{equation}
for almost every $z$.  If $z$ is such that \eqref{local} holds, then by the strong law of large numbers, we see that
$\frac{\sum_{i=1}^n \log |\lambda_i-z| }{n}$ converges almost surely to $\int_{\BBC} \log|w-z|\ d\mu(w)$.  This shows that (ii) implies (i) for almost sure convergence.  The proof for convergence in probability is identical and is left as an exercise to the reader.

Now we show that (iii) implies (ii) for almost sure convergence.  Let $z$ be such that \eqref{local} and (iii) hold.  To show (ii), it suffices from \eqref{det} to show that $\frac{1}{n} \sum_{i=1}^n \log \sigma_i$ converges almost surely to $\int_{\BBC} \log|w-z|\ d\mu(w)$, where $\sigma_i = \sigma_i( \frac{1}{\sqrt{n}} A_n - zI )$ are the singular values of $\frac{1}{\sqrt{n}} A_n - zI$.  On the other hand, from (iii) we already know that $\frac{1}{n} \sum_{i=1}^n \log \sqrt{\sigma^2_i + \eps_n}$ converges almost surely to $\int_{\BBC} \log|w-z|\ d\mu(w)$.  Thus it suffices to show that
\begin{equation}\label{seps}
\frac{1}{n} \sum_{i=1}^n \log \sqrt{\sigma^2_i + \eps_n} - \log \sigma_i
\end{equation}
converges almost surely to zero.

From Lemma \ref{tight}, we know that $\frac{1}{n^2} \|A_n\|_2^2$ is almost surely bounded, and so for each $z$
$$ \frac{1}{n} \sum_{i=1}^n \sigma_i^2 = \frac{1}{n} \| \frac{1}{\sqrt{n}} A_n - z I \|_2^2$$
is almost surely bounded also.  From this we easily see that
$$\frac{1}{n} \sum_{1 \leq i \leq n: \sigma_i \geq \delta_n} \log \sqrt{\sigma^2_i + \eps_n} - \log \sigma_i$$
converges almost surely to zero for some sequence $\delta_n$ (depending on $\eps_n$) converging sufficiently slowly to zero.  To conclude the almost sure convergence of \eqref{seps} to zero, it thus suffices to show that
$$\frac{1}{n} \sum_{1 \leq i \leq n: \sigma_i \leq \delta_n} \log \frac{1}{\sigma_i}$$
converges almost surely to zero.
Using Lemma \ref{lemma:lsv}, we almost surely have $\sup_i \log \frac{1}{\sigma_i} \leq O(\log n)$ for all but finitely many $n$, so it suffices to show that
$$\frac{1}{n} \sum_{1 \leq i \leq n - n^{0.99}: \sigma_i < \delta_n} \log \frac{1}{\sigma_i}.$$
converges almost surely to zero.  To do this, it suffices by the union bound and the Borel-Cantelli lemma to show that
\begin{equation}\label{eqn:lowersing}
 \P(  \sigma_{n-i} \leq c \frac{i}{n}  ) = O( \exp( - n^{0.01} ) ).
\end{equation}
for all $1 \leq i \leq n - n^{0.99}$ and some $c>0$ independent of $n$.

For this we argue as in the proof of Lemma \ref{slib}. Fix $i$. Let $A'_n$ be the matrix form by the first $n-k$ rows of
$A_n - z \sqrt n I$ with $k:=i/2$ and $\sigma'_j, 1\le j\le n-k$ be
the singular values of $A_n'$(in decreasing order, as usual).  By
the interlacing law (Lemma \ref{lemma:interlacing}) and
re-normalizing,

\begin{equation} \label{eqn:lowersing1} \sigma_{n-i} \ge \frac{1}{\sqrt n}
\sigma'_{n-i}. \end{equation}

By Lemma \ref{lemma:twosum}, we have that

$$\sigma_1'^{-2} + \dots + \sigma_{n-k}'^{-2} = \dist_1 ^{-2} +
\dots + \dist_{n-k}^{-2}, $$ where $\dist_j$ is the distance from
the $j$th row of $A'_n$ to the subspace spanned by the remaining
rows.

As shown in the proof of Lemma \ref{highdim}, with probability
$1-\exp(-n^{-0.01})$, $\dist_j$ is bounded from below by $\Omega
(\sqrt k) = \Omega (\sqrt i)$ for all $j$. Thus, with this
probability,  the right hand side in the above identity is $O(n/i)$.
On the other hand, as the $\sigma_j'$ are ordered decreasingly, the
left hand side is at least

$$(i-k) \sigma_{n-i}'^{-2} = \frac{i}{2} \sigma_{n-i}'^{-2}. $$

It follows that with probability $1-\exp(-n^{-0.01})$,

$$\sigma'_{n-i} = \Omega ( \frac{i}{\sqrt n} ). $$

This and \eqref{eqn:lowersing1} complete the proof of
\eqref{eqn:lowersing}, and so \eqref{seps} converges almost surely to zero.

As previously observed, the convergence of \eqref{seps} to zero shows that (ii) implies (iii) for almost sure convergence.  An inspection of the argument shows the convergence of \eqref{seps} to zero also lets us deduce (iii) from (ii).  The claim for convergence in probability follows similarly.  To conclude the proof of Theorem \ref{theorem:main4}, it thus suffices to show that (i) implies (ii).

Again we start with the almost sure convergence case.
Assume that (i) holds, and let $z$ be such that \eqref{local} holds.  By shifting $A$ by $\sqrt{n} z I$ if necessary we may take $z$ to be zero.  Let $\lambda_1,\ldots,\lambda_n$ denote the eigenvalues of $\frac{1}{\sqrt{n}} A_n$.  By \eqref{det}, it suffices to show that
$\frac{1}{n} \sum_{j=1}^n \log |\lambda_j|$ converges almost surely to $\int_{\BBC} \log|w|\ d\mu(w)$.  From \eqref{lam} we know that $\frac{1}{n}  \sum_{j=1}^n |\lambda_j|^2$ is almost surely bounded.  From this and (i) we conclude that $\frac{1}{n} \sum_{j=1}^n \log (|\lambda_j|+\eps)$ converges almost surely to $\int_{\BBC} (\log|w|+\eps)\ d\mu(w)$ for any fixed $\eps > 0$.  Combining this with \eqref{local} and dominated convergence, we see that $\frac{1}{n} \sum_{j=1}^n \log (|\lambda_j|+\eps_n)$ converges almost surely to $\int_{\BBC} \log|w|\ d\mu(w)$ for some sequence $\eps_n > 0$ converging sufficiently slowly to zero.  It thus suffices to show that
$$\frac{1}{n} \sum_{j=1}^n \log (|\lambda_j|+\eps_n) - \log |\lambda_j|$$
converges almost surely to zero.

By repeating the arguments used to establish the almost sure convergence of \eqref{seps} to zero, it suffices to show that
$$\frac{1}{n} \sum_{1 \leq i \leq n: |\lambda_i| \leq \delta_n} \log \frac{1}{|\lambda_i|}$$
converges almost surely to zero.

Let us order the eigenvalues $\lambda_i$ so that $|\lambda_1| \geq \ldots \geq |\lambda_n|$.  From Lemma \ref{lemma:lsv} and \eqref{eqn:lowersing} (and the Borel-Cantelli lemma) we know that we almost surely have
$$\frac{1}{n} \sum_{(1-\kappa) n < i \leq n} \log \frac{1}{\sigma_i} \leq O( \kappa \log \frac{1}{\kappa} )$$
for all but finitely many $n$ for any fixed $0 < \kappa < 1/2$, and hence by Weyl's comparison inequality (Lemma \ref{product}) that we almost surely have
$$\frac{1}{n} \sum_{(1-\kappa) n < i \leq n} \log \frac{1}{|\lambda_i|} \leq O(\kappa \log \frac{1}{\kappa} )$$
for all but finitely many $n$ also.  Since the left-hand side is bounded from below by $\kappa \log \frac{1}{| \lambda_{\lfloor (1-\kappa) n\rfloor}|}$ we almost surely conclude a lower bound of the form
$$ |\lambda_{\lfloor (1-\kappa) n\rfloor}| \geq \kappa^{O(1)}$$
for all but finitely many $n$.  In particular (by setting $\delta$ to be a suitable power of $\kappa$) this implies that almost surely
$$\frac{1}{n} \sum_{1 \leq i \leq n: |\lambda_i| \leq \delta} \log \frac{1}{|\lambda_i|} \leq O(\delta^c)$$
for all but finitely many $n$ for any fixed $0 < \delta \ll 1$ and some absolute constant $c > 0$, and the claim follows.  The analogous implication for convergence in probability is similar.  The proof of Theorem \ref{theorem:main4} is now complete.

\appendix

\section{Linear algebra inequalities}\label{remarks-sec}

In this appendix we record some elementary identities and inequalities regarding the eigenvalues and singular values of matrices.

\begin{lemma}[Cauchy's interlacing law]\label{lemma:interlacing}
Let $A$ be an $n \times n$ matrix with complex entries and $A'$
be the submatrix formed by the first $m:=n-k$ rows. Let $\sigma_1(A) \geq \ldots \geq \sigma_n(A) \geq 0$ denote the singular values of $A$, and similarly for $A'$.  Then we have
$$\sigma_i (A) \ge \sigma_i(A') \ge \sigma_{i+k}(A)$$
for every $1 \leq i \leq n-k$.
\end{lemma}

\begin{proof}  The claim follows easily from the minimax characterization
$$ \sigma_i(A) = \sup_{V_i \subset \BBC^n} \inf_{v \in V_i: \|v\|=1} \| Av_i \|$$
and
$$ \sigma_i(A') = \sup_{V_i \subset \BBC^{n-k}} \inf_{v \in V_i: \|v\|=1} \| Av_i \|$$
of the singular values, where $V_i$ range over $i$-dimensional complex subspaces.
\end{proof}

\begin{lemma}[Weyl comparison inequality for second moment]\label{compar}  Let $A = (a_{ij})_{1 \leq i,j \leq n} \in M_n(\BBC)$ have generalized eigenvalues $\lambda_1,\ldots,\lambda_n \in \BBC$ and singular values $\sigma_1(A) \geq \ldots \geq \sigma_n(A) \geq 0$.  Then
$$ \sum_{j=1}^n |\lambda_j|^2 \leq \sum_{j=1}^n \sigma_j(A)^2 = \|A\|_2^2 = \sum_{i=1}^n \sum_{j=1}^n |a_{ij}|^2.$$
\end{lemma}

\begin{proof} The two equalities here are clear, so it suffices to prove the inequality.  By the Jordan normal form we can write $A = BUB^{-1}$ for some upper-triangular $U$ and invertible $B$.  By the $QR$ factorization we can write $B = QR$ for some orthogonal $Q$ and upper triangular $R$.  We conclude that $A = QVQ^{-1}$ for some upper triangular $V$.  Conjugating by $Q$, we thus reduce to the case when $A$ is an upper triangular matrix, in which case the eigenvalues are simply the diagonal entries $a_{11},\ldots,a_{nn}$ and the claim is clear.
\end{proof}

We also have the following (stronger) variant of the above inequality:

\begin{lemma}[Weyl comparison inequality for products]\label{product}  Let $A = (a_{ij})_{1 \leq i,j \leq n} \in M_n(\BBC)$ have generalized eigenvalues $\lambda_1,\ldots,\lambda_n \in \BBC$, ordered so that $|\lambda_1| \leq \ldots \leq |\lambda_n|$, and singular values $\sigma_1(A) \geq \ldots \geq \sigma_n(A) \geq 0$.  Then we have
$$ \prod_{j=1}^J |\lambda_j| \leq \prod_{j=1}^J \sigma_j(A)$$
and
$$ \prod_{j=J}^n \sigma_j(A) \leq \prod_{j=J}^n |\lambda_j|$$
for all $0 \leq J \leq n$.
\end{lemma}

\begin{proof} It suffices to prove the former claim, as the latter then follows from \eqref{det}.  By arguing as in Lemma \ref{compar} we may assume that $A$ is upper triangular, so that the diagonal entries are some permutation of $\lambda_1,\ldots,\lambda_n$.  Consider the symmetric minor $A'$ of $A$ formed by the rows and columns corresponding to the entries $\lambda_1,\ldots,\lambda_J$.  The determinant of this matrix is then $\lambda_1 \ldots \lambda_J$, and thus by \eqref{det} we have
$$ \prod_{j=1}^J \sigma_j(A') = \prod_{j=1}^J |\lambda_j|.$$
The claim then follows from the Cauchy interlacing inequality (Lemma \ref{lemma:interlacing}).
\end{proof}

Now we record a useful identity for the \emph{negative} second moment of a rectangular matrix.

\begin{lemma}[Negative second moment]\label{lemma:twosum}  Let $1 \leq n' \leq n$, and let $A$ be a full rank $n' \times n$ matrix with singular values $\sigma_1(A) \geq \ldots \geq \sigma_{n'}(A) > 0$ and rows $X_1,\ldots,X_{n'} \in \BBC^n$.  For each $1 \leq i \leq n'$, let $W_i$ be the hyperplane generated by the $n'-1$ rows $X_1,\ldots,X_{i-1},X_{i+1},\ldots,X_{n'}$.  Then
$$ \sum_{j=1}^{n'} \sigma_j(A)^{-2} = \sum_{j=1}^{n'} \dist(X_j,W_j)^{-2}.$$
\end{lemma}

\begin{proof}  Observe that the $n' \times n'$ matrix $(AA^\ast)^{-1}$ has eigenvalues
$$\sigma_1(A)^{-2},\ldots,\sigma_{n'}(A)^{-2}. $$ Taking traces, we conclude that
$$ \sum_{j=1}^{n'} \sigma_j(A)^{-2} = \sum_{j=1}^{n'} (AA^\ast)^{-1} e_j \cdot e_j$$
where $e_1,\ldots,e_{n'}$ is the standard basis of $\BBC^{n'}$. But
if $v_j := (AA^\ast)^{-1} e_j = (v_{j,1},\ldots,v_{j,n'})$, then $A^\ast
v_j = v_{j,1} X_1 + \ldots + v_{j,n'} X_{n'}$ is orthogonal to $A^\ast
e_i = X_i$ for $i \neq j$ (and thus orthogonal to $W_j$), and has an
inner product of $1$ with $A^\ast e_j = X_j$.  Taking inner products of
$A^\ast v_j$ with the orthogonal projection of $X_j$ to $W_j$, we
conclude that
$$ v_{j,j} \dist(X_j,W_j)^2 = 1.$$
Since $v_{j,j} = v_j \cdot e_j = (AA^\ast)^{-1} e_j \cdot e_j$, the claim follows.
\end{proof}

\section{ A result of Dozier and Silverstein} \label{section:DS}

Here we reproduce Theorem 1.1 of \cite{doz} which we used in the end
of Section \ref{section:lowdim}.

\begin{theorem} \label{theorem:DS}\cite[Theorem 1.1]{doz} Let $c$ be a positive constant and  $\a$ be a random variable
with variance one. Let $X_n$ be an $n \times r$ random matrix
whose entries are iid copies of $\a$, where $r = (c+o(1))n$. Let $M_n$ be a random $n
\times r$ matrix independent from $X_n$ such that the ESD of $M_n
M_n^\ast$ converges to a limiting distribution $H$. Define $C_n:=
\frac{c}{n} (M_n + X_n) (M_n + X_n)^\ast$. Then the ESD of $C_n$
converges almost surely (and hence also in probability) to a limiting distribution $F$, whose Stieljes transform
$m(z):= \int \frac{1}{\lambda-z} d F(\lambda) $ satisfies the integral equation
\begin{equation} \label{eqn:DS} m = \int \frac{ d H(t)}{ \frac{t}{1+ cm} - (1+cm)z + (1-c)}
\end{equation}

\noindent for any $z \in \BBC$. \end{theorem}

\begin{remark} \label{remark:DS1}
The theorem still holds if we restrict the size $n$ of the matrices to
an infinite subsequence $n_1 < n_2 < \dots $ of positive integers.
One can show this by, for example, artificially filling in the missing
indices or repeat the proof of Theorem \ref{theorem:DS} under
this restriction.
\end{remark}

\begin{remark} \label{remark:DS2}
In \eqref{eqn:DS}, $H$ appears, but the actual definition of $M_n$
is irrelevant. Thus, one can conclude that if $M_n$ and $M_n'$ are
such that the ESD's of $M_n M_n^{\ast}$ and $M_n' M_n'^{\ast}$ tend
to the same limit, then the ESDs of $\frac{c}{n} (M_n + X_n) (M_n +
X_n)^\ast $ and $\frac{c}{n} (M_n' + X_n) (M_n' + X_n)^\ast$ also tend to
the same limit.
\end{remark}

\begin{remark} \label{remark:DS3}
It was mentioned by Speicher \cite{spe} and also Krishnapur
(private communication) that Theorem \ref{theorem:DS} can be proved
using free probability, which is different from the approach in
\cite{doz}.
\end{remark}

\section{Using a Hermitian invariance principle \\ (by Manjunath Krishnapur)}\label{krish-sec}

The authors have shown invariance principles for ESDs of several
non-Hermitian matrix models. As in earlier papers, the proof goes
through Hermitian matrices, but does not need rates of convergence
of the Hermitian ESDs, thanks to  new ideas such as
Lemma~\ref{highdim}. However, because of the use of
Theorem~\ref{theorem:DS}, it may appear that a limiting result for
the associated Hermitian matrices is necessary to carry the program
through. In this appendix, we point out how one may obtain a 
weak invariance principle for ESDs of non-Hermitian matrices by using an
invariance principle  for Hermitian matrices due to
Chatterjee~\cite{chatterjee}, in cases where a convergence result
such as Theorem~\ref{theorem:DS} is not available. As mentioned
earlier, other parts of the proof do not require the entries are
iid. Thus, as a consequence, we can obtain a weak invariance
principle for a random matrix model with independent but not
identically distributed entries.

We need the following definition from \cite[Section 2]{TV-circular}.

\begin{definition}[Controlled second moment]  Let $\kappa \geq 1$.
A complex random variable $\a$ is said to have
\emph{$\kappa$-controlled second moment} if one has the upper bound
$$ \E |\a|^2 \leq \kappa$$
(in particular, $|\E \a| \leq \kappa^{1/2}$), and the lower bound
\begin{equation}\label{eyot}
 \E \Re( z \a - w )^2 \I(|\a| \leq \kappa) \geq \frac{1}{\kappa} \Re(z)^2
\end{equation}
for all complex numbers $z,w$.
\end{definition}

{\it Example.} The Bernoulli random variable ($\P(\a=+1) = \P(\a=-1)
= 1/2$) has
 $1$-controlled second moment.  The condition \eqref{eyot} asserts in particular that
  $\a$ has variance at least $\frac{1}{\kappa}$, but also asserts that a significant portion of
  this variance occurs inside the event $|\a| \leq \kappa$, and also contains some more
  technical phase information about the covariance matrix of $\Re(\a)$ and $\Im(\a)$.

\begin{theorem}\label{thm:meanvar} Let $M_n=\left(\mu^{(n)}_{i,j}\right)_{i,j\le n}$
and $C_n=\left(\sigma^{(n)}_{i,j}\right)_{i,j\le n}$ be constant (i.e. deterministic) matrices  satisfying
\begin{enumerate}
\item $\sup_n n^{-2}\|M_n\|_2^2 <\infty$,
\item $a\le \sigma^{(n)}_{i,j} \le b$ for all $n,i,j$ for some $0<a<b<\infty$.
\end{enumerate}
Given a matrix  $\x=\left(x_{i,j}\right)_{i,j\le n}$ set
 \begin{equation*}
   A_n(\x)=\frac{1}{\sqrt{n}}\left( M_n+C_n\cdot \x \right) = \frac{1}{\sqrt{n}}
   \left(\mu_{i,j}^{(n)}+\sigma_{i,j}^{(n)}x_{i,j} \right)_{i,j\le n}.
 \end{equation*}
 (here "$\cdot$" denotes Hadamard product).

Now suppose that $x_{i,j}^{(n)}$ are independent complex-valued
random variables with $\E[x_{i,j}^{(n)}]=0$ and
$\E[|x_{i,j}^{(n)}|^2]=1$ and that  $y_{i,j}^{(n)}$ are independent
random variables, also having zero mean and unit variance. 

Assume furthermore that both $x_{ij}^{(n)}$ and $y_{ij}^{(n)}$ have $\kappa$-controlled second moment for some constant $\kappa >0$. 

Assume also Pastur's condition
\begin{equation}\label{eq:pastur}
 \frac{1}{n^2}\sum_{i,j=1}^n \E\left[ |x_{i,j}^{(n)}|^2 \I{|x_{i,j}^{(n)}|\ge \epsilon \sqrt{n}}\right] \longrightarrow 0 \hspace{1cm} \mb{ for all }\epsilon>0.\end{equation}
and the same for $\y$ in place of $\x$. Then,
\begin{equation*}
 \mu_{A_n(\x)} - \mu_{A_n(\y)} \rightarrow 0
\end{equation*}
in the sense of probability.
\end{theorem}
Some remarks.
\begin{enumerate}
\item If we assume that $x_{i,j}^{(n)}$ are i.i.d. and $y_{i,j}^{(n)}$ are i.i.d then Pastur's condition is obviously satisfied.  Further, the condition of $\kappa$-controlled second moment is also not necessary (see the first step in the proof sketch). 
\item Although the weak invariance principle in the paper uses only subsequential limits (see Remark \ref{detsec-remark}), it does use Theorem~{B.1} to say that subsequential limits are the same for $\x$ as for $\y$. Hence we need some changes in the proof in order to establish Theorem \ref{thm:meanvar}, which we do in this appendix.
\item This highlights the important new ideas of the paper, such as Lemma~\ref{highdim}, which eliminate the need for rates of convergence of  ESDs of the Hermitian matrices $(A_n-zI)^*(A_n-zI)$. This is unlike all earlier papers in the subject that followed Bai's approach and required such rates (eg., \cite{bai-first},\cite{TV-circular},\cite{gotze},\cite{PZ}).  The need for rates made it impossible to use the invariance principle for Hermitian matrices as we shall do now.
\item Take $C_n=J$ (all ones matrix) and $M_n=0$. Then Pastur's condition (\ref{eq:pastur}) implies almost sure convergence of the ESD of $A_n(\x)^*A_n(\x)$ (see \cite[Theorem 3.9]{bai}). For general $C_n$, since we use Chatterjee's invariance principle which assumes Pastur's condition but only gives weak invariance, we are able to assert only weak invariance for the non-Hermitian ESDs also. Thus, there is some room for improvement here, namely, to strengthen the conclusion of Theorem~\ref{thm:meanvar} to almost sure convergence.
\item Does ESD of $A_n(\x)$ converge? Perhaps so, provided the singular values of $C_n-zI$ have a limiting measure for every $z$. In \cite{KV} we have discussed some easy-to-check sufficient conditions on $C_n$ which implies convergence.  
\end{enumerate}
The following lemma is a ``Wishart'' analogue of the computations in section 2 of \cite{chatterjee} which considers Wigner matrices. As in that paper, the idea is to consider the Stieltjes transform of the ESD of $A_n(\x)^*A_n(\x)$ as a  function of $\x$. However a slight twist is needed as compared to Wigner matrices, because the entries of $A_n(\x)^*A_n(\x)$ are quadratic in $\x$ whereas the invariance principle we invoke requires bounds on the sup-norm of derivatives of the Stieltjes transform.
\begin{lemma}\label{lem:hemitianinvariance} Let $\x$ and $\y$ be as in Theorem~\ref{thm:meanvar}. Let $\nu_n^\x$ and $\nu_n^\y$ be the ESDs of $A_n(\x)^*A_n(\x)$ and $A_n(\y)^*A_n(\y)$.  Then $\nu_n^{\x}-\nu_n^{\y} \rightarrow 0$ weakly as $n\rightarrow \infty$.
\end{lemma}
\begin{proof} Let
\begin{equation*}
 H_n(\x) =  \left[ \begin{array}{cc} 0 & A_n(\x) \\ A_n(\x)^* & 0 \end{array} \right]
\end{equation*}
have ESD $\theta_n^\x$. The eigenvalues of $H_n(\x)$ are exactly the positive and negative square roots of the eigenvalues of $A_n(\x)^*A_n(\x)$. Thus we must show that $\theta_n^{\x}-\theta_n^{\y}\rightarrow 0$ weakly, in probability.  Fix any $\alpha$ in the upper half plane and let $f(\x) := \frac{1}{2n}\mb{Tr}(H_n(\x)-\alpha I)^{-1}$. The proof is complete if we show that $\E[f(\x)]-\E[f(\y)]\rightarrow 0$ for any $\alpha$ with $\Im\{\alpha\}>0$. This can be done by following the same calculations as in \cite{chatterjee}. It works because the entries of $H_n(\x)$ are linear in $\x$ and hence the first partial derivative of $H_n$ with respect to any $x_{i,j}$ is a constant matrix. One must also use the upper bound on $\sigma_{i,j}$ to bound the derivatives of $f$.
\end{proof}
\noindent{\bf Remark:} Obviously the same conclusion holds for $A_n-zI$, just by absorbing $zI$ into $M_n$.
\begin{proof}[Proof of Theorem~\ref{thm:meanvar}] The conditions on $M_n$ and $C_n$ show that
the first condition of Theorem~\ref{theorem:replacement} is satisfied
 (where the two matrices $A_n$ and $B_n$ are now $A_n(\x)$ and $A_n(\y)$).

 Thus we only need to show an analogue of Proposition~{2.2} (only the weak part). We sketch the modifications needed.
\begin{enumerate}
\item Lemma~\ref{lemma:lsv} can be proved under independence and
$\kappa$-controlled second moment  without i.i.d. assumption (see
\cite[Theorem 2.5]{TV-circular}). If we make i.i.d. assumption, then Lemma~\ref{lemma:lsv} is itself applicable, which explains the first remark after the statement of the theorem.

 The upper bounds on singular values in (\ref{bns}) are very general and hold in our setting for the same reasons.  Hence we reduce to Lemma~\ref{highdim} and Lemma~\ref{lowdim} as in the paper.

\item The high-dimensional contribution (analogue of Lemma~\ref{highdim})
is proved almost the same way. In the proof of the lower tail bound (Proposition~\ref{xw})
 use the bounds on $\sigma_{i,j}^{(n)}$ appropriately.
 In particular, we get a lower bounds of $a^2(n-d)$ for the second moment
 of $\mb{dist}(X,W)$ in Lemma~\ref{lemma:secondmomentofdist},
 and in applying Theorem~\ref{theorem:Talagrand} we get a Lipschitz constant of $b$ for $F(X)=\mb{dist}(X,W)$.

\item In the low-dimensional contribution (Lemma~\ref{lowdim}), the calculations in sections
\ref{section:larget}, \ref{section:moderatet} and \ref{section:smallt} are exactly
 as before (in section~\ref{section:moderatet}, we use the concentration result already outlined in the previous step).
\item That leaves section~\ref{section:intermediatet}, which is the only step that
is differently handled. Here we apply Lemma~\ref{lem:hemitianinvariance} instead of quoting Theorem~\ref{theorem:DS}.
\end{enumerate}
\end{proof}

{\it Acknowledgements.} The first author is supported by a grant
from the Macarthur Foundation and by NSF grant DMS-0649473. The
second author is supported by an NSF Career Grant. The authors
would like to thank M. Krishnapur for useful
discussions and his careful reading of an early draft, and Ken Miller, Ricky, and weiyu for further corrections.
We also like to thank P. Matchett Wood for providing the figures in the introduction.

\end{document}